\documentclass[a4paper,12pt]{amsart}
\usepackage{graphicx}
\usepackage{amsmath,amssymb,amsthm,amsfonts}
\usepackage{amssymb}
\usepackage[active]{srcltx}
\usepackage{geometry}
\usepackage{color}
\newtheorem{thm}{Theorem}[section]

\newtheorem{lem}[thm]{Lemma}
\newtheorem{prop}[thm]{Proposition}
\newtheorem{rem}[thm]{Remark}

\showoutput

\newcommand{\R}{\mathbb R}

\newcommand{\N}{\mathbb N}

\begin{document}
	
	\title[Regularity ]{Third order estimates and the regularity of the stress field for solutions to $p$-Laplace equations.}
	
	\author{Daniel Baratta, Berardino Sciunzi, Domenico Vuono}

	\email[Daniel Baratta]{daniel.baratta@unical.it}
	\email[Berardino Sciunzi]{sciunzi@mat.unical.it}
	\email[Domenico Vuono]{domenico.vuono@unical.it}
	\address{Dipartimento di Matematica e Informatica, Università della Calabria,
		Ponte Pietro Bucci 31B, 87036 Arcavacata di Rende, Cosenza, Italy}
	

	\thanks{B. Sciunzi and D. Vuono are members of INdAM. 
		The authors are supported by PRIN PNRR P2022YFAJH \em{Linear and Nonlinear PDEs: New directions and applications.} 
	}
	
	
	\begin{abstract}
		We consider solutions to
		$$ - \Delta_{p} u = f(x)  \quad \text{in } \Omega\, ,$$
		when $p$ approaches the semilinear limiting case  $p=2$ and we get  third order estimates. As a consequence we deduce  improved regularity properties  of the stress field.
	\end{abstract}
	
	\maketitle
	
	\section{Introduction}
	
	We deal with the regularity of the third derivatives  of weak solutions to:
	\begin{equation} \label{eq:problema}
		- \Delta_{p} u = f(x)  \quad \text{in } \Omega\,
	\end{equation}
	where $\Delta_{p} u := - \operatorname{div} (|\nabla u|^{p-2} \nabla u)$ is the $p$-Laplace operator and $\Omega $ is a domain of $\R^{n}$ with $n \ge 2$. 
	
	We consider possibly sign-changing solutions $u\in W^{1,p}_{loc}(\Omega)$ such that
	
	\begin{equation}\label{soluzionedebole}
		\int_\Omega |\nabla u|^{p-2}(\nabla u,\nabla \psi)\,dx=\int_\Omega f \psi \,dx \quad \forall \, \psi\in C^{\infty}_c(\Omega).
	\end{equation}
	
	\noindent In general solutions of \eqref{eq:problema} are not classical. It is well known (see \cite{DB,LIB,T}) that  solutions of \eqref{eq:problema} are of class $C^{1,\beta}_{loc}(\Omega)\cap C^2(\Omega\setminus Z_u)$, where $Z_u$ denotes the set where the gradient vanishes, for every $p>1$ and under suitable assumptions on the source term.

	\noindent Here we address the study of the integrability of the third derivatives of the solutions. The latter can be defined as distributional derivatives in a suitable meaning. Up to our knowledge  there are no other results in the literature in this direction. As a corollary, our  estimates  will allow us to infer regularity results for the stress field $|\nabla u|^{p-2}\nabla u$.

	\noindent The Calderón-Zygmund theory in the quasilinear context is a deep and challenging  issue, see \cite{1,2,3,4,6,7,9,10,12,13,16,17,22,23,24,25}. In particular we refer the readers to \cite{9,12,13,16,23}  for remarkable progress in this direction. Let us stress some very recent results in \cite{DM,DM2}. Very fine results regarding the Sobolev regularity of the stress field are in \cite{2,3,4,5,Cma,7,17, 22}. We also mention the regularity results in \cite{6} obtained with a different approach.

	\noindent It is well known that $u\in W^{2,2}_{loc}(\Omega)$, for $1<p<3$ (see \cite{DS,S1,S2}). 
This actually holds  for $f$ is $W^{1,1}_{loc}(\Omega)\cap L^q(\Omega)$, $q>n.$ In addition, the result can be extended on the boundary, when Dirichlet or Neumann boundary condition is imposed (see \cite{MMS}). For the vectorial case, recent developments are obtained in \cite{BaCiDiMa,DM,DM2,DM3,MMSV}.

	\noindent Our results are influenced by the classical theory of elliptic regularity, particularly by the constant $C(n,q)$ found in the Calderón-Zygmund estimate
	\begin{equation}\label{calderonintoduzione}
		\|D^{2}w\|_{L^{q}(\Omega)} \le C(n,q) \|\Delta w\|_{L^{q}(\Omega)}.
	\end{equation}
	
	\noindent Under the influence of this factor, we obtain second order  regularity results for the stress field $|\nabla u|^{p-2}\nabla u$:

	\begin{thm}\label{chestress}
		Let $\Omega \subset \R^{N}$ be a domain and $u \in C^{1,\beta}_{loc}(\Omega)$ be a weak solution of \eqref{eq:problema}.
		We set
		\begin{equation}\label{daf}   \begin{cases}
				q_N = 2(q_{N-1}-1) \\
				q_0 = 4\\
			\end{cases}
		\end{equation}
		Let  $q>q_0$, with $q_{N} < q \le q_{N+1}$ for some $N \in \N_0.$ \newline 
		Assume $p$ such that
		\begin{equation}\label{valorip}
			2-\frac{1}{C(n,q)}< p <\min \left\{2+\frac{1}{q-1} ,2+\frac{1}{C(n,q)}\right\}.
		\end{equation}
		Let $f(x) \in W^{2,q/(q-1)}_{loc}(\Omega) \cap C^{1,\beta'}_{loc}(\Omega)$ and 
		\[
		f\geq \tau >0 \quad in \quad  \Omega\,,\quad  (or\,\,f\leq \tau <0 \quad in \quad  \Omega\,)   .
		\]
		Then we have
		\begin{equation}\label{eq:campovett}
			|\nabla u|^{k-1}  \nabla u \in W^{2,1}_{loc}(\Omega,\R^{N})
		\end{equation}
		for any $k > (\alpha+1)/2$, where 
		\begin{equation}
			\alpha > \frac{(q-q_N)}{2^Nq}(1-p)+\frac{3-p}{2^N}+1.
		\end{equation}
		
		\noindent In case $f$ has no sign, let $q\ge q_0$, with $q_N\le q<q_{N+1}$ for some $N\in \N_0$, and $q_N$ given in \eqref{daf}. Assume $p$ satisfies \eqref{valorip}. Then  \eqref{eq:campovett} holds for any $k \ge (p+\alpha)/2,$
  where $\alpha > (3-p)/2^N+1$ (if $N=0$, $\alpha \ge 4-p$).

	\end{thm}
	\

 \

	\noindent In \cite{5} the authors proved that $|\nabla u|^{p-2}\nabla u\in W^{1,2}_{loc}(\Omega)$, if $f\in L^2_{loc}(\Omega).$
	The natural generalization could be 
	$$|\nabla u|^{p-2}\nabla u\in W^{1,\tilde \alpha}_{loc}(\Omega) \text{  iff  } f\in L_{loc}^{\tilde\alpha}(\Omega).$$
	Partial results have been obtained in \cite{17}, when $\tilde \alpha-2$ is sufficiently small, with $f\in L^{\tilde \alpha}(\Omega)\cap W^{-1,p'}(\Omega)$. Here, as a consequence of the regularity results on third derivatives, we obtain  improved regularity properties of the stress field when $p$ approaches the semilinear case $p=2$.

	\begin{thm}\label{stressfield}
		Let $\Omega \subset \R^{n}$ be a domain and $u \in C^{1,\beta}_{loc}(\Omega)$ be a weak solution of \eqref{eq:problema}. Assume that $\Tilde{\alpha} \ge 3$ and we set $q=2(\tilde \alpha-1)$. Suppose that  
		\begin{equation} \label{assunzionisup}
			\max\left\{2-\frac{1}{2\tilde\alpha-1},2-\frac{1}{C(n,q)} \right\}<p \le 2. 
		\end{equation}
		where $C(n,q)$ is given by \eqref{calderonintoduzione}. Let $f(x) \in W^{2,2(\tilde \alpha-1)/3}_{loc}(\Omega) \cap C^{1,\beta^{'}}_{loc}(\Omega)$. Then
		\begin{equation} \label{eq:tesi}
			|\nabla u|^{p-2} \nabla u  \in W^{1,\tilde\alpha}_{loc}(\Omega).
		\end{equation}
		
	\end{thm}
	
	\noindent To accomplish this, we first extend a result in \cite{MRS}, which does not take into account the effects of the boundary of $\Omega$. We have the following

	\begin{thm}\label{teoremacalderon}
		Let $\Omega$ be a domain of $\R^{n}$ and $u\in C^{1,\beta}_{loc}(\Omega)$ be a weak solution to \eqref{eq:problema}, with $f(x)\in W^{1,1}_{loc}(\Omega)\cap C_{loc}^{0,\beta'}(\Omega).$ Let $q \ge 2$ and $p$ be such that $$|p-2| < \frac{1}{C(n,q)},$$ with $C(n,q)$ given by \eqref{calderonintoduzione}. Then, if $1 < p \le 2$, we have $$u \in W^{2,q}_{loc}(\Omega).$$
		If $p>2$ then the same conclusion holds if $p$ satisfies, in addition,   $q<\frac{p-1}{p-2}.$ 
		
	\end{thm}

	\noindent Theorems \ref{chestress} and Theorem \ref{stressfield} follow from more general estimates concerning the third derivatives. More precisely we have  the following

	\begin{thm} \label{derivate terze}
		Let $\Omega\subset \mathbb{R}^n$ be a domain and let $u\in C^{1,\beta}_{loc}(\Omega)$ be a weak solution of \eqref{eq:problema}. 
		
		For $\gamma> \max \{(p-2)/(p-1),2-p\}$, we set 
		\begin{equation} \label{definizionediq_n1}
			\begin{cases}
				q_N = 2(q_{N-1}-1) \\
				q_0 = 3+\gamma.\\
			\end{cases}
		\end{equation}
		Let $q>q_0$, with $q_N< q\le q_{N+1}$ for some $N \in \mathbb{N}_0$ and $p$ be such that: 
		\begin{equation}\label{evvai}
			2-\frac{1}{C(n,q)}< p <\min \left\{2+\frac{1}{q-1} ,2+\frac{1}{C(n,q)}\right\},
		\end{equation}
		where $C(n,q)$ is given by \eqref{calderonintoduzione}. Assume \begin{equation}\label{valorealpha}
			\alpha>\alpha(p,q,N):=\frac{(q-q_N)}{2^Nq}(1-p)+\frac{3-p}{2^N}+1.
		\end{equation}
		Let $f(x)\in W^{2,q/(q-\gamma)}_{loc}(\Omega)\cap C^{1,\beta'}_{loc}(\Omega)$ and
		\[
		f\geq \tau >0 \quad in \quad  \Omega\,,\quad  (or\,\,f\leq \tau <0 \quad in \quad  \Omega\,)   .
		\]
		Then, for any $\tilde \Omega \subset \subset \Omega$, for $\gamma\ge 1$, we have that 
		\begin{equation}\label{stima derivata terza}
			\int_{\tilde \Omega \setminus Z_u} |\nabla u|^{p-2+\alpha} |D^2u|^{\gamma-1}|D^3u|^{2} \, dx \le C,
		\end{equation}
		where $C =C(\gamma,\alpha,p,q,n,f,\tilde \Omega,\|\nabla u\|_{L^{\infty}_{loc}(\Omega)},N)$.
		
		If $ \max \{(p-2)/(p-1),2-p\}<\gamma<1$, for any $\tilde \Omega \subset \subset \Omega$, we deduce that
		\begin{equation}\label{stima derivata terza33}
			\int_{\tilde \Omega \setminus (Z_u\cup \{|D^2u|=0\})} |\nabla u|^{p-2+\alpha} |D^2u|^{\gamma-1}|D^3u|^{2} \, dx \le C,
		\end{equation}
		where $C =C(\gamma,\alpha,p,q,n,f,\tilde \Omega,\|\nabla u\|_{L^{\infty}_{loc}(\Omega)},N)$.
	\end{thm}

	\noindent We also point out that, if we assume that $f$ has no a sign, Theroem \ref{derivate terze} holds, in the case $\alpha > (3-p)/2^N+1$, for some $N\in \N$, if $N=0$, $\alpha \ge 4-p$ (see Theorem \ref{derivateterze2} in Section \ref{risultatiprincipali}). 
	\
	
	\
	
	\noindent Our paper is structured as follows. In Section \ref{notations} we state the notation and we recall some preliminary results that will be useful to prove our result. In Section \ref{risultatiprincipali} we prove Theorem \ref{chestress}, Theorem \ref{stressfield}, Theorem \ref{teoremacalderon} and Theorem \ref{derivate terze}.

	\section{Notations and Preliminary results}\label{notations}
	\noindent {\bf Notation.} Generic fixed and numerical constants will
	be denoted by $C$ (with subscript in some case) and they will be
	allowed to vary within a single line or formula. By $|A|$ we will
	denote the Lebesgue measure of a measurable set $A$.

	For $a, b \in \R^n$ we denote by $a \otimes b$ the matrix whose entries are $(a \otimes b)_{ij}=a_ib_j$. We remark that for any $v,w \in \R^n$ it holds that:
	$$\langle a \otimes b \ v, w \rangle = \langle b, v \rangle \langle a , w \rangle.$$

	In the sequel, we will use the notation
	$u_i := u_{x_i}$ , for $i = 1, . . . , n$, to indicate the partial derivative of $u$ with respect to $x_i$.  The second derivatives of $u$ will be
	denoted with $u_{ij},$ $i,j = 1,...,n$. 
	
	Let $x_0\in \Omega$ and  $B_{2R}(x_0)\subset\subset\Omega$. Now we consider the regularized problem
	\begin{equation} \label{eq:problregol}
		\begin{cases}
			-\operatorname{div}\left( (\epsilon +|\nabla u_{\epsilon}|^{2})^{\frac{p-2}{2}} \, \nabla u_{\epsilon} \right) = f(x) & \text{in } B_{2R}(x_0) \\
			u_{\epsilon} = u & \text{on } \partial B_{2R}(x_0),
		\end{cases}
	\end{equation}
	where $\epsilon \in [0,1)$. The existence of such a weak solution follows by classical minimization procedure. Notice that by standard regularity results \cite{DiKaSc,GT} the solution $ u_\varepsilon$ belongs to $C^2(B_{2R}(x_0))$, therefore $u_\varepsilon$ is a classical solution of 
	\begin{equation}\label{uepsrisolve}
		-\Delta u_\varepsilon=(p-2)\frac{(D^2u_\varepsilon \nabla u_\varepsilon,\nabla u_\varepsilon)}{(\varepsilon+|\nabla u_\varepsilon|^2)}+ \frac{f}{(\varepsilon+|\nabla u_{\varepsilon}|^2)^{\frac{p-2}{2}}}.
	\end{equation}
	Let us consider a compact set $K\subset \subset B_{2R}(x_0)$. We note that, by \cite{DiKaSc}, ${u}_\varepsilon$ is uniformly bounded in $C^{1,\beta}(K)$, for some $0<\beta<1$. Moreover, by Dirichlet datum in \eqref{eq:problregol} and by uniqueness, it follows that
	\begin{equation}\label{evvai2}
		{u}_\varepsilon\rightarrow  u \quad \text{in the norm }\|\cdot\|_{C^{1,\beta}(K)}.
	\end{equation}
	\
	
	\

	We remind (see e.g. \cite{GT}) that it holds the Calder\'on-Zygmund inequality for elliptic operators 
	
	\begin{lem}\label{Zygmund}
		Let $\Omega$ be a bounded smooth domain of $\R^n$ and let $w \in W^{2,q}_{0}(\Omega),$ then there exists a positive constant $C(n,q)$ such that 
		\begin{equation} \label{eq:CZ}
			\|D^{2}w\|_{L^{q}(\Omega)} \le C(n,q) \|\Delta w\|_{L^{q}(\Omega)}.
		\end{equation}
	\end{lem}

	We recall a weighted Hessian regularity result and an integrability property of the inverse of the weight $|\nabla u|^{p-2}$.
	
	\begin{thm}[see \cite{DS,MRS}]\label{teosecond}
		Let $\Omega$ be a domain of $\R^n$ and $p>1$. Let $u_\varepsilon$ be a weak solution of \eqref{eq:problregol}, with $f(x)\in W^{1,1}_{loc}(\Omega)\cap C_{loc}^{0,\beta'}(\Omega).$ Then, for any $\tilde{\Omega} \subset\subset \Omega$,  it follows that 
		\begin{equation}\label{der2}
			\int_{\tilde{\Omega}}(\varepsilon+|\nabla u_\varepsilon|^2)^{\frac{p-2-\beta}{2}}\|D^2{u_\varepsilon}\|^2\,dx\leq C
		\end{equation}
		for any $0\le \beta, \beta'<1$, with   $C=C(p,\beta,\tilde \Omega,f,\|\nabla u\|_{L^{\infty}_{loc}(\Omega)})$. 
		
		Moreover, for $p>2$ and for any fixed $1\le q<\frac{p-1}{p-2}$, there exists a constant $C=C(p,q,\tilde \Omega,f,\|\nabla u\|_{L^{\infty}_{loc}(\Omega)})>0$ such that 
		\begin{equation}\label{mezzoinversodelpeso}
			\int_{\tilde \Omega} \frac{f^2}{(\varepsilon+|\nabla u_\varepsilon|^2)^{\frac{q(p-2)}{2}}}\le C.
		\end{equation}
		If $1<p\le 2$, the inequality \eqref{mezzoinversodelpeso} holds for any $1\le q<+\infty.$
	\end{thm}
	\
	
	\

	We conclude this section recalling a result  that will be very useful in the proofs
	of our results.
	
	\begin{thm}[see \cite{DS}]\label{inversodelpeso}
		Let $\Omega \subset \R^{n}$ be domain, $p>1$ and $u \in C^{1,\beta}_{loc}(\Omega)$ be a weak solution to \eqref{eq:problema} with $f \in W^{1,1}_{loc}(\Omega)\cap C^{0,\beta '}_{loc}(\Omega)$ and
		\[
		f\geq \tau >0 \quad in \quad  \Omega\,,\quad  (or\,\,f\leq \tau <0 \quad in \quad  \Omega\,)   .
		\] Then, for any $\Tilde{\Omega} \subset \subset \Omega$, we have that $|Z_u| = 0$ and $$\int_{\tilde \Omega} \frac{1}{|\nabla u|^{(p-1)r}} \, dx \le C,$$
		for any $r<1$ with $C=C(p,r,\tilde \Omega,f,\|\nabla u\|_{L^{\infty}_{loc}(\Omega)})$ positive constant. 
	\end{thm}
	
	\section{Main result}\label{risultatiprincipali}
	
	\noindent We start the section by extending a result proven in \cite{MRS}, for solutions $u\in W_{loc}^{1,p}(\Omega)$.

	\begin{proof}[Proof of Theorem \ref{teoremacalderon}]
		We start with the case $p>2$.
		Let $x_0 \in \Omega$ and $R>0$ be fixed such that $B_{2R}(x_0) \subset \subset \Omega.$  Let $\varphi\in C_c^\infty(B_{2R}(x_0))$ such that $\varphi=1$ in $B_R(x_0)$, $\varphi=0$ in $(B_{2R}(x_0))^c$ and $|\nabla \varphi|\leq\frac{2}{R}$ in the annular $B_{2R}(x_0)\setminus B_R(x_0)$. 
		
		\noindent Since $u_\varepsilon\varphi\in W_0^{2,q}(B_{2R}(x_0))$, using Lemma \ref{Zygmund}, and recalling that $u_\varepsilon$ is solution of \eqref{uepsrisolve}, we have 
		\begin{equation}
			\begin{split}
				& \|D^{2}(u_{\epsilon}\varphi)\|_{L^{q}(B_{2R}(x_0))}  \le C(n,q) \|\Delta(u_{\epsilon}\varphi)\|_{L^{q}(B_{2R}(x_0))}
				\\ &= C(n,q) \|\varphi\Delta u_\varepsilon+2(\nabla u_\varepsilon,\nabla \varphi)+u_\varepsilon\Delta\varphi \|_{L^{q}(B_{2R}(x_0))} 
				\\ &\le C(n,q)\left\| (p-2)\frac{(\varphi D^2u_\varepsilon \nabla u_\varepsilon,\nabla u_\varepsilon)}{(\epsilon + |\nabla u_{\epsilon}|^{2})} + \frac{\varphi f}{(\epsilon + |\nabla u_{\epsilon}|^{2})^{\frac{p-2}{2}}}\right\|_{L^{q}(B_{2R}(x_0))}
				\\ &+ C(n,q,R,\| \nabla u\|_{L^{\infty}_{loc}(\Omega)}),
			\end{split}
		\end{equation}
		where in the last inequality we have used the fact that $|\nabla u_\varepsilon|\le C$, for a constant $C$ not depending on $\varepsilon$ (see \cite{DiKaSc}), and $C(n,q,R,\| \nabla u\|_{L^{\infty}_{loc}(\Omega)})$ is 
		a positive constant.
		
		Since $\varphi D^{2}u_{\epsilon} =D^{2}(u_{\epsilon} \, \varphi) -  \nabla u_{\epsilon} \otimes \nabla \varphi  -\nabla \varphi\otimes\nabla u_{\epsilon}-  u_{\epsilon}D^{2}\varphi $, and using Theorem \ref{teosecond}, in particular see \eqref{mezzoinversodelpeso}, we deduce that 
		\begin{equation}
			\begin{split}
				\|D^{2}(u_{\epsilon}\varphi)\|_{L^{q}(B_{2R}(x_0))} & \le C(n,q)(p-2)\|D^{2}(u_{\epsilon}\varphi)\|_{L^{q}(B_{2R}(x_0))}
				\\ & + C(n,q,R,f,p,\| \nabla u\|_{L^{\infty}_{loc}(\Omega)}),
			\end{split}
		\end{equation}
		where $C(n,q,R,f,p,\| \nabla u\|_{L^{\infty}(\Omega)})$ is a positive constant.
		
		Since $p-2 < \frac{1}{C(n,q)}$ and using the fact that 
		\begin{equation}
			\|D^{2}u_{\epsilon}\|_{L^{q}(B_{R}(x_0))}\le \|D^{2}(u_{\epsilon}\varphi)\|_{L^{q}(B_{2R}(x_0))},
		\end{equation}
		we infer that
		
		\begin{equation}
			\|D^{2}u_{\epsilon}\|_{L^{q}(B_{R}(x_0))}\le C(n,q,R,f,p,\| \nabla u\|_{L^{\infty}_{loc}(\Omega)}).
		\end{equation}
		
		\noindent where $C(n,q,R,f,p,\| \nabla u\|_{L^{\infty}_{loc}(\Omega)})$ is a positive constant not depending on $\epsilon.$ 
		\newline Therefore we have
		\begin{equation*}
			\sup_{\epsilon > 0} \|u_{\epsilon}\|_{W^{2,q}(B_R(x_0))} < \infty.
		\end{equation*}
		By Rellich's theorem, up to a subsequence we get
		\begin{equation*}
			u_{\epsilon} \rightharpoonup w \in W^{2,q}(B_R(x_0))
		\end{equation*}
		and almost everywhere in $B_R(x_0).$ Therefore we get
		\begin{equation*}
			u \equiv w \in W^{2,q}(B_R(x_0)).
		\end{equation*}

	\end{proof}
	Next we shall be intersted in the second linearized equation of \eqref{soluzionedebole}
	at any fixed solution $u$, which we can write as follows. For $\varphi \in C^{\infty}_c(\Omega \setminus Z_u)$, taking $\psi =\varphi_{ij}:=(\partial ^2\varphi)/(\partial x_i\partial x_j)$ in \eqref{soluzionedebole}, we get 
	
	\begin{eqnarray} \label{eq:lin sec}
		\nonumber && \int_{\Omega} |\nabla u|^{p-2} (\nabla u_{ij},\nabla \varphi) + (p-2) \int_{\Omega} |\nabla u|^{p-4} (\nabla u,\nabla u_{j})(\nabla u_{i}, \nabla \varphi) \\
		\nonumber &+& (p-2)(p-4) \int_{\Omega} |\nabla u|^{p-6} (\nabla u,\nabla u_{j})(\nabla u_{i},\nabla u)(\nabla u,\nabla \varphi) \\
		\nonumber &+& (p-2) \int_{\Omega} |\nabla u|^{p-4} (\nabla u_{ij},\nabla u)(\nabla u,\nabla \varphi)\\
		\nonumber &+& (p-2) \int_{\Omega} |\nabla u|^{p-4} (\nabla u_{i},\nabla u_{j})(\nabla u,\nabla \varphi) \\
		\nonumber &+& (p-2) \int_{\Omega} |\nabla u|^{p-4} (\nabla u_{i},\nabla u)(\nabla u_{j},\nabla \varphi) = \int_{\Omega} f_{ij} \, \varphi, \\
	\end{eqnarray}
	
	where $f_{ij}:=(\partial ^2 f)/(\partial x_i\partial x_j).$
	\
	
	\
	
	Now we are ready to prove

	\begin{proof}[Proof of Theorem \ref{derivate terze}]
		
		\noindent To facilitate reading, we split the case  $\gamma\ge 1$ from the case $\max\{(p-2)/(p-1),2-p\}<\gamma<1$. 
		
		\textbf{Case $\gamma \ge 1$.}
		
		We start exploiting the second linearized equation \eqref{eq:lin sec}. We use a regularization argument. For $\epsilon > 0$ we define:
		\begin{equation}\label{gepsilon}
			G_{\epsilon}(t) := 
			\begin{cases}
				0               & \text{if} \quad t \in [0, \epsilon]\\
				(2t -2\epsilon) & \text{if} \quad t \in [\epsilon, 2\epsilon]\\
				t               & \text{if} \quad t \in [2\epsilon, \infty)\\
			\end{cases}
		\end{equation}
		if $t > 0$, while $G_{\epsilon}(t) := -G_{\epsilon}(-t)$ if $t \le 0.$ Let $x_0\in \Omega$ and $B=B_R(x_0)\subset \subset \Omega$. Let us consider the cut-off function $\psi_{R} := \psi \in C^{\infty}_{c}(B_{2R}(x_0))$ such that: \begin{equation} \label{eq:psi}
			\begin{cases}
				\psi = 1 & \text{in} \quad B_{R}(x_0)\\
				|\nabla \psi| \le \frac{2}{R} & \text{in} \quad B_{2R}(x_0)\setminus B_{R}(x_0)\\
				\psi = 0 & \text{in} \quad B_{2R}(x_0)^{c}.
			\end{cases}
		\end{equation} 
  Now, let us define 
   \begin{equation}\label{funzionetest}
		\varphi := \frac{G_{\epsilon}(|\nabla u|)^{\alpha+1}}{|\nabla u|} \, |D^2u|^{\gamma-1}u_{ij} \, \psi^{2}.
	\end{equation} 
  So 
		\begin{equation*}
			\begin{split}
				\nabla \varphi &=  |D^2u|^{\gamma-1}\frac{G_\varepsilon(|\nabla u|)^{\alpha+1}}{|\nabla u|}\psi^2 \nabla u_{ij}\\
    &+(\gamma-1)\frac{G_\varepsilon(|\nabla u|)^{\alpha+1}}{|\nabla u|}\psi^2 u_{ij}|D^2u|^{\gamma-3}D^3u\cdot D^2u \\
    &+ G_\varepsilon (|\nabla u|)^{\alpha}((\alpha+1)G'_\varepsilon(|\nabla u|)-G_\varepsilon(|\nabla u|)|\nabla u|^{-1})|D^2u|^{\gamma-1}u_{ij}\psi^2|\nabla u|^{-2} D^2 u\nabla u \\
    &+2\frac{G_\varepsilon(|\nabla u|)^{\alpha+1}}{|\nabla u|}|D^2u|^{\gamma-1}u_{ij}\psi \nabla \psi,\\
			\end{split}
		\end{equation*}
		
		where $D^3u\cdot D^2u$ denotes the vector $(...,\sum_{k,l}u_{klj}u_{kl},...)$, for $j=1,...,n$.

		Testing \eqref{funzionetest} in \eqref{eq:lin sec} we get 
		\begin{equation}\label{eq:sost}
			\begin{split}
				&\int_{\Omega} |\nabla u|^{p-2} G_{\epsilon}(|\nabla u|)^{\alpha+1}|\nabla u|^{-1} |D^2u|^{\gamma-1}|\nabla u_{ij}|^{2} \psi^{2} \, dx\\
            &+(\gamma-1)\int_{\Omega} |\nabla u|^{p-2} G_{\epsilon}(|\nabla u|)^{\alpha+1}|\nabla u|^{-1} |D^2u|^{\gamma-3}(\nabla u_{ij} u_{ij},D^3u\cdot D^2u) \psi^{2} \, dx\\
            &+  \int_{\Omega} |\nabla u|^{p-2} G_\varepsilon (|\nabla u|)^{\alpha}((\alpha+1)G'_\varepsilon(|\nabla u|)-G_\varepsilon(|\nabla u|)|\nabla u|^{-1}) \times\\
                 &  \qquad \qquad \qquad \qquad\times(\nabla u_{ij}, D^2 u\nabla u) |\nabla u|^{-2} |D^2u|^{\gamma-1}u_{ij} \psi^{2} \, dx \\
                 &+ 2 \int_{\Omega} |\nabla u|^{p-2} G_{\epsilon}(|\nabla u|)^{\alpha+1}|\nabla u|^{-1}  (\nabla u_{ij},\nabla \psi) |D^2u|^{\gamma-1}u_{ij} \psi \, dx \\
                 &+ (p-2) \int_{\Omega} |\nabla u|^{p-4}(\nabla u,\nabla u_j)(\nabla u_i, \nabla u_{ij})|D^2u|^{\gamma-1}  G_{\epsilon}(|\nabla u|)^{\alpha+1}|\nabla u|^{-1}   \psi^{2} \, dx \\
				&+ (\gamma-1)(p-2) \int_{\Omega} |\nabla u|^{p-4}(\nabla u,\nabla u_j)(\nabla u_i, D^3u\cdot D^2u)u_{ij}\times\\ &\qquad \qquad \qquad \qquad\times|D^2u|^{\gamma-3}  G_{\epsilon}(|\nabla u|)^{\alpha+1}|\nabla u|^{-1}   \psi^{2} \, dx \\
				&+ (p-2) \int_{\Omega} |\nabla u|^{p-4}(\nabla u,\nabla u_j)(\nabla u_i, D^2 u\nabla u)|\nabla u|^{-2}\times \\ & \qquad  \qquad \times  G_\varepsilon (|\nabla u|)^{\alpha}((\alpha+1)G'_\varepsilon(|\nabla u|)-G_\varepsilon(|\nabla u|)|\nabla u|^{-1}) |D^2u|^{\gamma-1}u_{ij} \psi^{2} \, dx  \\
				&+ 2(p-2) \int_{\Omega} |\nabla u|^{p-4}(\nabla u,\nabla u_j)(\nabla u_i, \nabla \psi)  G_{\epsilon}(|\nabla u|)^{\alpha+1}|\nabla u|^{-1} |D^2u|^{\gamma-1} u_{ij} \psi \, dx \\
				&+ (p-2)(p-4) \int_{\Omega} |\nabla u|^{p-6} (\nabla u,\nabla u_j)(\nabla u_i,\nabla u)(\nabla u,\nabla u_{ij})\times \\ &\qquad \qquad\qquad\qquad \times |D^2u|^{\gamma-1}G_{\epsilon}(|\nabla u|)^{\alpha+1}|\nabla u|^{-1} \psi^{2} \, dx \\
    &+ (\gamma-1)(p-2)(p-4) \int_{\Omega} |\nabla u|^{p-6} (\nabla u,\nabla u_j)(\nabla u_i,\nabla u)\times \\   & \qquad \qquad \qquad  \qquad\times(\nabla u,D^3 u\cdot D^2u) u_{ij}|D^2u|^{\gamma-3}G_{\epsilon}(|\nabla u|)^{\alpha+1}|\nabla u|^{-1}  \psi^{2} \, dx \\
            \end{split}
        \end{equation}    
        
        \begin{equation*}
            \begin{split}
				&+ (p-2)(p-4) \int_{\Omega} |\nabla u|^{p-6} (\nabla u,\nabla u_j)(\nabla u_i,\nabla u) (\nabla u, D^2 u\nabla u)|\nabla u|^{-2} \times 
				\\ & \qquad \qquad \qquad\times  G_\varepsilon (|\nabla u|)^{\alpha}((\alpha+1)G'_\varepsilon(|\nabla u|)-G_\varepsilon(|\nabla u|)|\nabla u|^{-1}) |D^2u|^{\gamma-1} u_{ij} \psi^{2} \, dx \\    
				&+ 2(p-2)(p-4) \int_{\Omega} |\nabla u|^{p-6} (\nabla u,\nabla u_j)(\nabla u_i,\nabla u)(\nabla u,\nabla \psi) \times \\ & \qquad\qquad\qquad\qquad\times G_{\epsilon}(|\nabla u|)^{\alpha+1}|\nabla u|^{-1}  |D^2u|^{\gamma-1}u_{ij}\psi \, dx \\
				&+ (p-2) \int_{\Omega} |\nabla u|^{p-4} (\nabla u_{ij}, \nabla u)^{2} |D^2u|^{\gamma-1}G_{\epsilon}(|\nabla u|)^{\alpha+1}|\nabla u|^{-1} \psi^{2} \, dx \\
				&+ (\gamma-1)(p-2) \int_{\Omega} |\nabla u|^{p-4} (\nabla u_{ij}u_{ij}, \nabla u)(D^3u\cdot D^2u, \nabla u)\times \\& \qquad\qquad\qquad\times|D^2u|^{\gamma-3}G_{\epsilon}(|\nabla u|)^{\alpha+1}|\nabla u|^{-1} \psi^{2} \, dx \\
				&+  (p-2) \int_{\Omega} |\nabla u|^{p-4} (\nabla u_{ij},\nabla u)(\nabla u,D^2 u\nabla u)|\nabla u|^{-2} \times 
				\\ & \qquad  \qquad\times  G_\varepsilon (|\nabla u|)^{\alpha}((\alpha+1)G'_\varepsilon(|\nabla u|)-G_\varepsilon(|\nabla u|)|\nabla u|^{-1}) |D^2u|^{\gamma-1} u_{ij}\psi^{2} \, dx\\
				&+ 2(p-2) \int_{\Omega}|\nabla u|^{p-4} (\nabla u_{ij},\nabla u)(\nabla u,\nabla \psi) G_{\epsilon}(|\nabla u|)^{\alpha+1}|\nabla u|^{-1}    |D^2u|^{\gamma-1}u_{ij}\psi \, dx\\
				&+ (p-2) \int_{\Omega} |\nabla u|^{p-4} (\nabla u_i,\nabla u_j)(\nabla u,\nabla u_{ij}) |D^2u|^{\gamma-1}G_{\epsilon}(|\nabla u|)^{\alpha+1}|\nabla u|^{-1}  \psi^{2} \, dx\\
				&+ (\gamma-1)(p-2) \int_{\Omega} |\nabla u|^{p-4} (\nabla u_i,\nabla u_j)(\nabla u,D^3u\cdot D^2u)\times \\&\qquad\qquad\qquad\qquad\times u_{ij} |D^2u|^{\gamma-3}G_{\epsilon}(|\nabla u|)^{\alpha+1}|\nabla u|^{-1}  \psi^{2} \, dx\\
				&+  (p-2) \int_{\Omega} |\nabla u|^{p-4} (\nabla u_i,\nabla u_j)(\nabla u,D^2 u\nabla u)|\nabla u|^{-2} \times 
				\\ &  \qquad \qquad \qquad\times  G_\varepsilon (|\nabla u|)^{\alpha}((\alpha+1)G'_\varepsilon(|\nabla u|)-G_\varepsilon(|\nabla u|)|\nabla u|^{-1}) |D^2u|^{\gamma-1}u_{ij}\psi^{2} \, dx \\
				&+ 2(p-2) \int_{\Omega} |\nabla u|^{p-4} (\nabla u_i,\nabla u_j)(\nabla u,\nabla \psi) G_{\epsilon}(|\nabla u|)^{\alpha+1}|\nabla u|^{-1} |D^2u|^{\gamma-1}u_{ij}\psi \, dx\\
				&+ (p-2) \int_{\Omega} |\nabla u|^{p-4} (\nabla u_i,\nabla u)(\nabla u_j,\nabla u_{ij}) |D^2 u|^{\gamma-1}G_{\epsilon}(|\nabla u|)^{\alpha+1}|\nabla u|^{-1} \psi^{2} \, dx\\
				&+ (\gamma-1)(p-2) \int_{\Omega} |\nabla u|^{p-4} (\nabla u_i,\nabla u)(\nabla u_j,D^3u\cdot D^2u)\times \\ & \qquad\qquad\qquad\qquad\times u_{ij} |D^2u|^{\gamma-3}G_{\epsilon}(|\nabla u|)^{\alpha+1}|\nabla u|^{-1} \psi^{2} \, dx\\
				&+  (p-2) \int_{\Omega} |\nabla u|^{p-4} (\nabla u_i,\nabla u)(\nabla u_j,D^2 u\nabla u)|\nabla u|^{-2} \times 
				\\ &  \qquad \qquad \qquad\times  G_\varepsilon (|\nabla u|)^{\alpha}((\alpha+1)G'_\varepsilon(|\nabla u|)-G_\varepsilon(|\nabla u|)|\nabla u|^{-1}) |D^2u|^{\gamma-1}u_{ij}\psi^{2} \, dx \\
				&+ 2(p-2) \int_{\Omega} |\nabla u|^{p-4} (\nabla u_i,\nabla u)(\nabla u_j,\nabla  \psi) G_{\epsilon}(|\nabla u|)^{\alpha+1}|\nabla u|^{-1} |D^2u|^{\gamma-1}u_{ij}\psi \, dx\\
				& - \int_{\Omega} f_{ij} G_{\epsilon}(|\nabla u|)^{\alpha+1}|\nabla u|^{-1} |D^2u|^{\gamma-1} u_{ij} \psi^{2} \, dx =:\tilde I_1+\cdot\cdot\cdot+\tilde I_{25}=0
			\end{split}
		\end{equation*} 
		
		Since $G_\varepsilon(t)\le t$ and $G'(t)\le 2$ for any $t>0$, we deduce that 
		\begin{equation}\label{eq:sost2}
            \begin{split}
			& \tilde I_1+\tilde I_2+\tilde I_{13}+\tilde I_{14}= \int_{\Omega} |\nabla u|^{p-2} G_{\epsilon}(|\nabla u|)^{\alpha+1}|\nabla u|^{-1} |D^2u|^{\gamma-1}|\nabla u_{ij}|^{2} \psi^{2} 
			\\
            &+ (\gamma-1)\int_{\Omega} |\nabla u|^{p-2} G_{\epsilon}(|\nabla u|)^{\alpha+1}|\nabla u|^{-1} |D^2u|^{\gamma-3}(\nabla u_{ij} u_{ij},D^3u\cdot D^2u) \psi^{2} \, dx\\
            &+(p-2) \int_{\Omega} |\nabla u|^{p-4} (\nabla u_{ij}, \nabla u)^{2}|D^2u|^{\gamma-1}G_{\epsilon}(|\nabla u|)^{\alpha+1}|\nabla u|^{-1}\psi^{2}\,dx   \\
            &+ (\gamma-1)(p-2) \int_{\Omega} |\nabla u|^{p-4} (\nabla u_{ij}u_{ij}, \nabla u)(D^3u\cdot D^2u, \nabla u)\times\\  &\qquad\qquad\qquad\qquad\times |D^2u|^{\gamma-3}G_{\epsilon}(|\nabla u|)^{\alpha+1}|\nabla u|^{-1}\psi^{2} \, dx
			\\
            &\le C(\alpha,p,\gamma)\int_{\Omega} |\nabla u|^{p-3}G_\varepsilon(|\nabla u|)^{\alpha}  |D^{2}u|^{\gamma+1} |D^3u|  \psi^{2} \,dx \\
		   &+ C(p) \int_{\Omega} |\nabla u|^{p-2}G_\varepsilon(|\nabla u|)^{\alpha}   |D^3u| |\nabla \psi| |D^2 u|^{\gamma} \psi\,dx \\
           &+ C(\alpha,p) \int_{\Omega\setminus Z_{u}} |\nabla u|^{p+\alpha-4} |D^{2}u|^{\gamma+3}  \psi^{2}\,dx \\
			 &+ C(p) \int_{\Omega\setminus Z_{u}} |\nabla u|^{p+\alpha-3} |D^2 u|^{\gamma+2} |\nabla \psi| \psi \,dx\\
			& \qquad     +  \int_{\Omega}  |\nabla u|^{\alpha} |f_{ij}||D^2 u|^{\gamma} \psi^{2}\,dx,
            \end{split}
        \end{equation}

		where $C(p), C(\alpha,p)$ and $ C(\alpha,p,\gamma)$ are positive constants.
  
		Summing on $i,j=1,...,n$ we get 
		
		\begin{equation*}\label{eq:sost232}
        \begin{split}
			& \Bar{I}_1+\Bar{I}_2+\Bar{I}_{13}+\Bar{I}_{14}= \int_{\Omega} |\nabla u|^{p-2} G_{\epsilon}(|\nabla u|)^{\alpha+1}|\nabla u|^{-1} |D^2u|^{\gamma-1}|D^3 u|^{2} \psi^{2} 
			\\  
            &+ (\gamma-1)\int_{\Omega} |\nabla u|^{p-2} G_{\epsilon}(|\nabla u|)^{\alpha+1}|\nabla u|^{-1} |D^2u|^{\gamma-3}|D^3u\cdot D^2u|^2 \psi^{2} \, dx
			\\ 
            &+(p-2) \int_{\Omega} |\nabla u|^{p-4} \sum_{i,j=1}^n(\nabla u_{ij}, \nabla u)^{2}|D^2u|^{\gamma-1}G_{\epsilon}(|\nabla u|)^{\alpha+1}|\nabla u|^{-1}\psi^{2}\,dx   \\  
            &+ (\gamma-1)(p-2) \int_{\Omega} |\nabla u|^{p-4} (D^3u\cdot D^2u, \nabla u)^2 |D^2u|^{\gamma-3}G_{\epsilon}(|\nabla u|)^{\alpha+1}|\nabla u|^{-1}\psi^{2} \, dx
			\\  
            &\le C(n,\alpha,p,\gamma)\int_{\Omega} |\nabla u|^{p-3}G_\varepsilon(|\nabla u|)^{\alpha}   |D^{2}u|^{\gamma+1} |D^3u|  \psi^{2} \,dx \\
            &+ C(n,p) \int_{\Omega} |\nabla u|^{p-2} G_\varepsilon(|\nabla u|)^{\alpha}  |D^3u| |\nabla \psi| |D^2 u|^{\gamma} \psi\,dx \\
        \end{split}
        \end{equation*}
        
        \begin{equation}            
        \begin{split}
            &+ C(n,\alpha,p) \int_{\Omega\setminus Z_{u}} |\nabla u|^{p+\alpha-4} |D^{2}u|^{\gamma+3}  \psi^{2}\,dx \\
			&+ C(n,p) \int_{\Omega\setminus Z_{u}} |\nabla u|^{p+\alpha-3} |D^2 u|^{\gamma+2} |\nabla \psi| \psi \,dx\\
			&+  C(n)\int_{\Omega}  |\nabla u|^{\alpha} |D^2f||D^2 u|^{\gamma} \psi^{2}\,dx,
        \end{split}
		\end{equation}
		where $C(n),C(n,p), C(n,\alpha,p)$ and $ C(n,\alpha,p,\gamma)$ are positive constants.

		In the case $p\ge 2$ and $\gamma\ge 1$, by Cauchy-Schwarz inequality we have
		
		\begin{equation}\label{a_2}
			\begin{split}        
				&\Bar{I}_1+\Bar{I}_2+\Bar{I}_{13}+\Bar{I}_{14} \\
				&\ge \int_{\Omega} |\nabla u|^{p-2} G_{\epsilon}(|\nabla u|)^{\alpha+1}|\nabla u|^{-1}|D^2u|^{\gamma-1} |D^3u|^{2} \psi^{2}\,dx.
			\end{split}
		\end{equation}

		Using Cauchy-Schwarz inequality for $1<p<2$ and $\gamma\ge 1$, we note that
		
		\begin{equation}\label{similvettoriale}
			\begin{split}
				\Bar{I}_{14}=&(\gamma-1)(p-2) \int_{\Omega} |\nabla u|^{p-4} (D^3u\cdot D^2u, \nabla u)^2 |D^2u|^{\gamma-3}G_{\epsilon}(|\nabla u|)^{\alpha+1}|\nabla u|^{-1}\psi^{2} \, dx\\ \ge &(\gamma-1)(p-2)\int_{\Omega} |\nabla u|^{p-2} G_{\epsilon}(|\nabla u|)^{\alpha+1}|\nabla u|^{-1} |D^2u|^{\gamma-3}|D^3u\cdot D^2u|^2 \psi^{2} \, dx.
			\end{split}
		\end{equation}
		Therefore, by \eqref{similvettoriale} and  using Cauchy-Schwarz inequality on $\Bar{I}_{13}$, we get 
		\begin{equation}\label{a_4}
			\begin{split}
				&\Bar{I}_1+\Bar{I}_2+\Bar{I}_{13}+\Bar{I}_{14} \\
				&\ge (p-1)\int_{\Omega} |\nabla u|^{p-2} G_{\epsilon}(|\nabla u|)^{\alpha+1}|\nabla u|^{-1}|D^2u|^{\gamma-1} |D^3u|^{2} \psi^{2}\,dx.
			\end{split}
		\end{equation}
		
		Using  \eqref{a_2} and \eqref{a_4} in \eqref{eq:sost2}, and since  $\gamma\ge 1$,  there exists a positive constant $C(p)$ such that  
		
		\begin{equation}\label{eq:sost4}
			\begin{split}
				C(p)&\int_{\Omega} |\nabla u|^{p-2} G_{\epsilon}(|\nabla u|)^{\alpha+1}|\nabla u|^{-1} |D^2u|^{\gamma-1}| D^3 u|^{2} \psi^{2} \,dx 
				\\ 
            &\le
C(n,\alpha,p,\gamma)\int_{\Omega} |\nabla u|^{p-3} G_\varepsilon(|\nabla u|)^{\alpha}  |D^{2}u|^{\gamma+1} |D^3u|  \psi^{2} \,dx \\            
            &+C(n,p)\int_{\Omega} |\nabla u|^{p-2}G_\varepsilon(|\nabla u|)^{\alpha}   |\nabla \psi||D^2u|^{\gamma} |D^3 u| \psi\,dx \\ 
            \end{split}
            \end{equation}
            \begin{equation*}
            \begin{split}
            &+C(n,\alpha,p) \int_{\Omega \setminus Z_u} |\nabla u|^{p+\alpha-4} |D^{2}u|^{\gamma+3}  \psi^{2}\,dx \\  
				&+C(n,p) \int_{\Omega \setminus Z_u} |\nabla u|^{p+\alpha-3} |D^2 u|^{\gamma+2} |\nabla \psi| \psi \,dx\\
				&+C(n)\int_{\Omega}  |\nabla u|^{\alpha} |D^2f||D^2 u|^{\gamma} \psi^{2}\,dx=: I_1+I_2+I_3+I_4+I_5.
			\end{split}
		\end{equation*}

		\noindent Now we estimate the right hand side of \eqref{eq:sost4}. We estimate the term $I_3$. Using a standard Young inequality we obtain 
		\begin{equation}
			\begin{split}
				I_3 & \le C(n,\alpha,p)\int_{B_{2R}\setminus Z_u} |\nabla u|^{p+\alpha-4} |D^{2}u|^{3+\gamma}\,dx   
				\\ &=C(n,\alpha,p)\int_{B_{2R}\setminus Z_u} |\nabla u|^{p+\alpha-4} |\nabla u|^{\frac{p-2-\beta}{2}} |\nabla u|^{\frac{2-p+\beta}{2}} |D^{2}u| |D^{2}u|^{2+\gamma}\, dx  
				\\ &\le \frac{C(n,\alpha,p)}{2} \int_{B_{2R}\setminus Z_u} |\nabla u|^{p-2-\beta} |D^{2}u|^{2} \,dx \\
				&+ \frac{C(n,\alpha,p)}{2} \int_{B_{2R}\setminus Z_u} |\nabla u|^{2(p-4+\alpha) +2-p+\beta} |D^{2}u|^{4+2\gamma} \,dx 
				\\ &\le C(n,\alpha,p,R,\beta,f, \|\nabla u\|_{L^{\infty}_{loc}(\Omega)}) \\
				&+C(n,\alpha,p) \int_{B_{2R}\setminus Z_u} |\nabla u|^{2(p-4+\alpha) +2-p+\beta} |D^{2}u|^{4+2\gamma} \,dx,   
			\end{split}
		\end{equation}

		\noindent where in the last inequality we have used Theorem \ref{teosecond} and \newline $C(n,\alpha,p,R,\beta,f, \|\nabla u\|_{L^{\infty}_{loc}(\Omega)})$ is a positive constant. Iterating this procedure $N$-times, we get 
		\begin{equation}
			\begin{split}
				I_3& \le NC(n,\alpha,p,R,\beta,f, \|\nabla u\|_{L^{\infty}_{loc}(\Omega)})
				\\ &+C(n,\alpha,p)\int_{B_{2R}\setminus Z_u} |\nabla u|^{2^N(p-4+\alpha) +\sum_{k=1}^{N}2^{k-1}(2-p+\beta)} |D^{2}u|^{q_N} \,dx, 
			\end{split}
		\end{equation}
		
		where $q_N$ is given by 
		
		\begin{equation} \label{definizionediq_n}
			\begin{cases}
				q_N = 2(q_{N-1}-1) \\
				q_0 = 3+\gamma.\\
			\end{cases}
		\end{equation}
		
		Now by Young's inequality with exponents $\left( q/q_N, q/( q-q_N)\right)$ we get 
		
		\begin{equation}
			\begin{split}
				\int_{B_{2R}\setminus Z_u} &|\nabla u|^{2^N(p-4+\alpha) +\sum_{k=1}^{N}2^{k-1}(2-p+\beta)} |D^{2}u|^{q_N} \,dx \\ 
				&\leq C( q)\int_{B_{2R}\setminus Z_u} |\nabla u|^{\frac{ q}{ q-q_N}\left(2^N(p-4+\alpha) +\sum_{k=1}^{N}2^{k-1}(2-p+\beta)\right)} \,dx \\
    &+ C( q) \int_{B_{2R}} |D^2 u|^{q}\,dx\\
			\end{split}
		\end{equation}
		
		where $C(q)$ is a positive constant.
		
		If 
		\begin{equation}\label{zio}
			\frac{ q}{q-q_N}\left(2^N(p-4+\alpha) +\sum_{k=1}^{N}2^{k-1}(2-p+\beta)\right)>-p+1,
		\end{equation}
		and $q>q_N$, from Theorem \ref{inversodelpeso}, we have 
		\begin{equation}\label{h_12}
			\int_{B_{2R}} |\nabla u|^{\frac{ q}{ q-q_N}\left(2^N(p-4+\alpha) +\sum_{k=1}^{N}2^{k-1}(2-p+\beta)\right)} \,dx \le C(N,\alpha,p,q,f,\beta,R,\|\nabla u\|_{L^{\infty}_{loc}(\Omega)}),
		\end{equation}
		where $C(N,\alpha,p,q,f,\beta,R,\|\nabla u\|_{L^{\infty}_{loc}(\Omega)})$ is a positive constant. 
		Moreover, since $q>q_N$, from Theorem \ref{teoremacalderon} we deduce that 
		\begin{equation}\label{h_22}
			\int_{B_{2R}}|D^2 u|^{ q}\,dx \le C(n,p,q,f,R),
		\end{equation}
		where $C(n,p,q,f,R)$ is a positive constant.
		
		By \eqref{h_12} and \eqref{h_22} we estimate 
		\begin{equation}\label{stimaI_32}
			I_3 \le C(\alpha,p,q,n,f,R,\beta,\|\nabla u\|_{L^{\infty}_{loc}(\Omega)},N),
		\end{equation}
		
		where $C(\alpha,p,q,n,f,R,\beta,\|\nabla u\|_{L^{\infty}_{loc}(\Omega)},N)$ is a positive constant.
		
		We note that \eqref{zio} yields, for $\beta$ close to $1$, 
		$$\alpha>\alpha (p,q,N):=\alpha_N:=\frac{(q-q_N)}{2^Nq}(1-p)+\frac{3-p}{2^N}+1.$$
		We note that, if $p<2+1/(q-1)$, that is $q<(p-1)/(p-2)$ we have,  $$\alpha_k<\alpha_{k-1}\quad \forall k\in \N,$$
		and $$\alpha >\alpha_N >1 \quad \forall N\in \N_0.$$

		For the term $I_1$, using weighted Young's inequality and by \eqref{stimaI_32} we get 
		
		\begin{equation}\label{stimaI_1}
			\begin{split}
				I_1&\le \theta\int_{\Omega} |\nabla u|^{p-2}G_{\epsilon}(|\nabla u|)^{\alpha+1}|\nabla u|^{-1} |D^2u|^{\gamma-1}|D^3u|^2\psi^2\,dx\\ &+C(n,\gamma,p,\alpha,\theta)\int_{\Omega\setminus Z_{u}} |\nabla u|^{p+\alpha-4} |D^2 u|^{3+\gamma} \psi^2\,dx 
				\\ &\le \theta\int_{\Omega} |\nabla u|^{p-2}G_{\epsilon}(|\nabla u|)^{\alpha+1}|\nabla u|^{-1} |D^2u|^{\gamma-1}|D^3u|^2\psi^2\,dx \\        &+C(\theta,\gamma,\alpha,p,q,n,f,R,\beta,\|\nabla u\|_{L^{\infty}_{loc}(\Omega)},N),
			\end{split}
		\end{equation}
		where $C(\theta,\gamma,\alpha,p,q,n,f,R,\beta,\|\nabla u\|_{L^{\infty}_{loc}(\Omega)},N)$ is a positive constant. 
		
		For the term $I_2$, since $\alpha>1$, by \eqref{eq:psi} and from Theorem \ref{teoremacalderon} we deduce that
		
		\begin{equation}\label{stimaI_2}
			\begin{split}
				I_2&\le \theta\int_{\Omega} |\nabla u|^{p-2}G_{\epsilon}(|\nabla u|)^{\alpha+1}|\nabla u|^{-1} |D^2u|^{\gamma-1}|D^3u|^2\psi^2\,dx
				\\ &+C(n,p,\theta)\int_{\Omega} |\nabla u|^{p+\alpha-2} |D^2 u|^{1+\gamma} |\nabla\psi|^2\,dx 
				\\ & \le \theta\int_{\Omega} |\nabla u|^{p-2}G_{\epsilon}(|\nabla u|)^{\alpha+1}|\nabla u|^{-1} |D^2u|^{\gamma-1}|D^3u|^2\psi^2\,dx
				\\ &+C(n,p,\theta,R,\|\nabla u\|_{L^{\infty}_{loc}(\Omega)})\int_{B_{2R}}|D^2u|^{1+\gamma}\,dx
				\\ & \le \theta\int_{\Omega} |\nabla u|^{p-2}G_{\epsilon}(|\nabla u|)^{\alpha+1}|\nabla u|^{-1} |D^2u|^{\gamma-1}|D^3u|^2\psi^2\,dx \\ &+C(\theta,\alpha,p,q,n,f,R,\|\nabla u\|_{L^{\infty}_{loc}(\Omega)},N),
			\end{split}
		\end{equation}
		where $C(\theta,\alpha,p,q,n,f,R,\|\nabla u\|_{L^{\infty}_{loc}(\Omega)},N)$ is a positive constant. 
		From Theorem \ref{teosecond}, by \eqref{stimaI_32}, \eqref{stimaI_2} and using standard Young's inequality we estimate
		\begin{equation}\label{stimaI_4}
			\begin{split}
				I_4&\le C(n,p,R)\int_{B_{2R}}|\nabla u|^{p+\alpha-3}|D^2u|^{2+\gamma}\,dx
				\\ & \le\frac{C(n,p,R)}{2}\int_{B_{2R}}|\nabla u|^{p+\alpha-2}|D^2u|^{1+\gamma}\,dx\\
    &+\frac{C(n,p,R)}{2}\int_{B_{2R}}|\nabla u|^{p+\alpha-4}|D^2u|^{3+\gamma}\,dx
				\\ & \le  C(\alpha,p,q,n,f,R,\|\nabla u\|_{L^{\infty}_{loc}(\Omega)},N),
			\end{split}
		\end{equation}
		
		where $C(\alpha,p,q,n,f,R,\|\nabla u\|_{L^{\infty}_{loc}(\Omega)},N)$ is a positive constant.
		
		For the last term $I_5$, by assumptions on $f$ and from Theorem \ref{teoremacalderon} we get
		\begin{equation}\label{stimaI_5}
			\begin{split}
				I_5\le C(n,\alpha,R,&\|\nabla u\|_{L^{\infty}_{loc}(\Omega)})\int_{B_{2R}}|D^2f||D^2u|^{\gamma}\,dx \\ &\le C(\alpha,R,f,n,p,q,\|\nabla u\|_{L^{\infty}_{loc}(\Omega)}),
			\end{split}
		\end{equation}
		
		where $C(\alpha,R,f,n,p,q,\|\nabla u\|_{L^{\infty}_{loc}(\Omega)})$ is a positive constant.
		
		Using \eqref{stimaI_32}, \eqref{stimaI_1}, \eqref{stimaI_2}, \eqref{stimaI_4} and \eqref{stimaI_5} in \eqref{eq:sost4}, we get: 
		
		\begin{equation} 
			\begin{split}			(C(p)-2\theta)&\int_{\Omega} |\nabla u|^{p-2}G_{\epsilon}(|\nabla u|)^{\alpha+1}|\nabla u|^{-1}|D^2u|^{\gamma-1} |D^3u|^{2} \psi^{2}\,dx\\
			&\le  C(\theta,\alpha,\gamma,p,q,n,f,R,\|\nabla u\|_{L^{\infty}_{loc}(\Omega)},N),\\
			\end{split}
		\end{equation}
		
		where $C(\theta,\alpha,\gamma,p,q,n,f,R,\|\nabla u\|_{L^{\infty}_{loc}(\Omega)},N)$ is a positive constant.
		
		Now, by Fatou's Lemma and choosing $\theta$ sufficiently small such that \newline
		$C(p) -2\theta > 0,$ 
		we get:
		
		\begin{equation}
			\int_{B_{R}(x_0)\setminus Z_{u}} |\nabla u|^{p-2+\alpha} |D^2u|^{\gamma-1}|D^3u|^{2} \, dx \le C(\theta,\gamma,\alpha,p,q,n,f,R,\|\nabla u\|_{L^{\infty}_{loc}(\Omega)},N).
		\end{equation}
		Using a covering argument, we are done.
		
		\textbf{Case $\max\{(p-2)/(p-1),2-p\}<\gamma<1$.}
		
		Let us define \begin{equation}\label{funzionetest77}
			\varphi := \frac{G_{\epsilon}(|\nabla u|)^{\alpha+1}}{|\nabla u|} \, G_\tau(|D^2u|)|D^2u|^{\gamma-2}u_{ij} \, \psi^{2},
		\end{equation}
		where $G_\tau,G_\varepsilon,\psi$ are defined in \eqref{gepsilon} and \eqref{eq:psi}. So we obtain 
		\begin{equation}
			\begin{split}
				\nabla \varphi &=  G_\tau(|D^2u|)|D^2u|^{\gamma-2}\frac{G_\varepsilon(|\nabla u|)^{\alpha+1}}{|\nabla u|}\psi^2 \nabla u_{ij}\\ &+\frac{G_\varepsilon(|\nabla u|)^{\alpha+1}}{|\nabla u|}\psi^2 u_{ij}|D^2u|^{\gamma-4}((\gamma-2)G_\tau(|D^2u|)+G'_\tau(|D^2u|)|D^2u|)D^3u\cdot D^2u \\&+ G_\varepsilon (|\nabla u|)^{\alpha}((\alpha+1)G'_\varepsilon(|\nabla u|)-G_\varepsilon(|\nabla u|) |\nabla u|^{-1})\times\\& \qquad\qquad\qquad\times G_\tau(|D^2u|)|D^2u|^{\gamma-2}u_{ij}\psi^2|\nabla u|^{-2} D^2 u\nabla u 
				\\&+2\frac{G_\varepsilon(|\nabla u|)^{\alpha+1}}{|\nabla u|}G_\tau(|D^2u|)|D^2u|^{\gamma-2}u_{ij}\psi \nabla \psi,
			\end{split}
		\end{equation}
		
		\noindent where $D^3u\cdot D^2u$ denotes the vector $(...,\sum_{k,l}u_{klj}u_{kl},...)$, for $j=1,...,n$. Now we set $h_\tau(t):=(\gamma-2)G_\tau(t)+G'_{\tau}(t)t$, for $t\ge 0.$
		Testing \eqref{funzionetest77} in \eqref{eq:lin sec} we get 
  
		\begin{equation}\label{eq:sost77}
			\begin{split}
				&  \int_{\Omega} |\nabla u|^{p-2} G_{\epsilon}(|\nabla u|)^{\alpha+1}|\nabla u|^{-1} G_\tau(|D^2u|)|D^2u|^{\gamma-2}|\nabla u_{ij}|^{2} \psi^{2} \, dx  \\ 
				&  +\int_{\Omega} |\nabla u|^{p-2} G_{\epsilon}(|\nabla u|)^{\alpha+1}|\nabla u|^{-1} |D^2u|^{\gamma-4}h_\tau(|D^2u|)(\nabla u_{ij} u_{ij},D^3u\cdot D^2u) \psi^{2} \, dx\\ 
				&+  \int_{\Omega} |\nabla u|^{p-2} G_\varepsilon (|\nabla u|)^{\alpha}((\alpha+1)G'_\varepsilon(|\nabla u|)-G_\varepsilon(|\nabla u|)|\nabla u|^{-1}) \times 
				\\ & \qquad \qquad \qquad \qquad \qquad\times(\nabla u_{ij}, D^2 u\nabla u) |\nabla u|^{-2} G_\tau(|D^2u|)|D^2u|^{\gamma-2}u_{ij} \psi^{2} \, dx \\
    &+ 2 \int_{\Omega} |\nabla u|^{p-2} G_{\epsilon}(|\nabla u|)^{\alpha+1}|\nabla u|^{-1}  (\nabla u_{ij},\nabla \psi) G_\tau(|D^2u|)|D^2u|^{\gamma-2}u_{ij} \psi \, dx \\
				&+ (p-2) \int_{\Omega} |\nabla u|^{p-4}(\nabla u,\nabla u_j)(\nabla u_i, \nabla u_{ij})G_\tau(|D^2u|)|D^2u|^{\gamma-2} \times\\
				& \qquad \qquad \qquad \qquad \qquad \times G_{\epsilon}(|\nabla u|)^{\alpha+1}|\nabla u|^{-1}   \psi^{2} \, dx \\
				&+ (p-2) \int_{\Omega} |\nabla u|^{p-4}(\nabla u,\nabla u_j)(\nabla u_i, D^3u\cdot D^2u)u_{ij}\times\\ &\qquad \qquad \qquad \qquad \qquad\times|D^2u|^{\gamma-4}h_\tau(|D^2u|)  G_{\epsilon}(|\nabla u|)^{\alpha+1}|\nabla u|^{-1}   \psi^{2} \, dx \\
				&+ (p-2) \int_{\Omega} |\nabla u|^{p-4}(\nabla u,\nabla u_j)(\nabla u_i, D^2 u\nabla u)|\nabla u|^{-2} G_\varepsilon (|\nabla u|)^{\alpha} \times \\ & \qquad  \qquad \times  ((\alpha+1)G'_\varepsilon(|\nabla u|)-G_\varepsilon(|\nabla u|)|\nabla u|^{-1}) G_\tau(|D^2u|)|D^2u|^{\gamma-2}u_{ij} \psi^{2} \, dx  \\
				&+ 2(p-2) \int_{\Omega} |\nabla u|^{p-4}(\nabla u,\nabla u_j)(\nabla u_i, \nabla \psi)  G_{\epsilon}(|\nabla u|)^{\alpha+1}|\nabla u|^{-1} \times \\
				& \qquad \qquad \qquad \qquad \qquad \times G_\tau(|D^2u|)|D^2u|^{\gamma-2} u_{ij} \psi \, dx \\
				&+ (p-2)(p-4) \int_{\Omega} |\nabla u|^{p-6} (\nabla u,\nabla u_j)(\nabla u_i,\nabla u)(\nabla u,\nabla u_{ij})\times \\ &\qquad \qquad\qquad\qquad \times G_\tau(|D^2u|)|D^2u|^{\gamma-2}G_{\epsilon}(|\nabla u|)^{\alpha+1}|\nabla u|^{-1} \psi^{2} \, dx \\
				&+ (p-2)(p-4) \int_{\Omega} |\nabla u|^{p-6} (\nabla u,\nabla u_j)(\nabla u_i,\nabla u)(\nabla u,D^3 u\cdot D^2u) \times \\   & \qquad \qquad \qquad  \qquad\times  u_{ij}|D^2u|^{\gamma-4}h_\tau(|D^2u|)G_{\epsilon}(|\nabla u|)^{\alpha+1}|\nabla u|^{-1}  \psi^{2} \, dx \\
				&+ (p-2)(p-4) \int_{\Omega} |\nabla u|^{p-6} (\nabla u,\nabla u_j)(\nabla u_i,\nabla u) (\nabla u, D^2 u\nabla u)|\nabla u|^{-2}G_\varepsilon (|\nabla u|)^{\alpha} \times 
				\\ & \qquad \qquad\times  ((\alpha+1)G'_\varepsilon(|\nabla u|)-G_\varepsilon(|\nabla u|)|\nabla u|^{-1}) G_\tau(|D^2u|)|D^2u|^{\gamma-2} u_{ij} \psi^{2} \, dx \\
				&+ 2(p-2)(p-4) \int_{\Omega} |\nabla u|^{p-6} (\nabla u,\nabla u_j)(\nabla u_i,\nabla u)(\nabla u,\nabla \psi) \times \\ & \qquad\qquad\qquad\qquad\times G_{\epsilon}(|\nabla u|)^{\alpha+1}|\nabla u|^{-1}  G_\tau(|D^2u|)|D^2u|^{\gamma-2}u_{ij}\psi \, dx \\
				&+ (p-2) \int_{\Omega} |\nabla u|^{p-4} (\nabla u_{ij}, \nabla u)^{2} G_\tau(|D^2u|)|D^2u|^{\gamma-2}G_{\epsilon}(|\nabla u|)^{\alpha+1}|\nabla u|^{-1} \psi^{2} \, dx \\
            \end{split}
            \end{equation}
            \begin{equation*}
            \begin{split}
            &+ (p-2) \int_{\Omega} |\nabla u|^{p-4} (\nabla u_{ij}u_{ij}, \nabla u)(D^3u\cdot D^2u, \nabla u)\times \\& \qquad\qquad\qquad\qquad\times|D^2u|^{\gamma-4}h_\tau(|D^2u|)G_{\epsilon}(|\nabla u|)^{\alpha+1}|\nabla u|^{-1} \psi^{2} \, dx \\
				&+  (p-2) \int_{\Omega} |\nabla u|^{p-4} (\nabla u_{ij},\nabla u)(\nabla u,D^2 u\nabla u)|\nabla u|^{-2} G_\varepsilon (|\nabla u|)^{\alpha} \times 
				\\ & \qquad  \qquad\times  ((\alpha+1)G'_\varepsilon(|\nabla u|)-G_\varepsilon(|\nabla u|)|\nabla u|^{-1}) G_\tau(|D^2u|)|D^2u|^{\gamma-2} u_{ij}\psi^{2} \, dx\\
    &+ 2(p-2) \int_{\Omega}|\nabla u|^{p-4} (\nabla u_{ij},\nabla u)(\nabla u,\nabla \psi) G_{\epsilon}(|\nabla u|)^{\alpha+1}|\nabla u|^{-1} \times \\
				& \qquad \qquad \qquad \qquad \times G_\tau(|D^2u|)|D^2u|^{\gamma-2}u_{ij}\psi \, dx\\
				&+ (p-2) \int_{\Omega} |\nabla u|^{p-4} (\nabla u_i,\nabla u_j)(\nabla u,\nabla u_{ij}) G_\tau(|D^2u|)|D^2u|^{\gamma-2} \times \\
				& \qquad \qquad \qquad \qquad G_{\epsilon}(|\nabla u|)^{\alpha+1}|\nabla u|^{-1}  \psi^{2} \, dx\\
			&+ (p-2) \int_{\Omega} |\nabla u|^{p-4} (\nabla u_i,\nabla u_j)(\nabla u,D^3u\cdot D^2u)\times \\&\qquad\qquad\qquad\qquad\times u_{ij} |D^2u|^{\gamma-4}h_\tau(|D^2u|)G_{\epsilon}(|\nabla u|)^{\alpha+1}|\nabla u|^{-1}  \psi^{2} \, dx\\
				&+  (p-2) \int_{\Omega} |\nabla u|^{p-4} (\nabla u_i,\nabla u_j)(\nabla u,D^2 u\nabla u)|\nabla u|^{-2} G_\varepsilon (|\nabla u|)^{\alpha}\times 
				\\ &   \qquad \qquad\times  ((\alpha+1)G'_\varepsilon(|\nabla u|)-G_\varepsilon(|\nabla u|)|\nabla u|^{-1}) G_\tau(|D^2u|)|D^2u|^{\gamma-2}u_{ij}\psi^{2} \, dx \\
				&+ 2(p-2) \int_{\Omega} |\nabla u|^{p-4} (\nabla u_i,\nabla u_j)(\nabla u,\nabla \psi) G_{\epsilon}(|\nabla u|)^{\alpha+1}|\nabla u|^{-1} \times\\
				&\qquad \qquad \qquad \qquad \times G_\tau(|D^2u|)|D^2u|^{\gamma-2}u_{ij}\psi \, dx\\
				&+ (p-2) \int_{\Omega} |\nabla u|^{p-4} (\nabla u_i,\nabla u)(\nabla u_j,\nabla u_{ij})G_\tau(|D^2u|)|D^2u|^{\gamma-2} \times \\
				&\qquad \qquad \qquad \qquad \times G_{\epsilon}(|\nabla u|)^{\alpha+1}|\nabla u|^{-1} \psi^{2} \, dx\\
				&+ (p-2) \int_{\Omega} |\nabla u|^{p-4} (\nabla u_i,\nabla u)(\nabla u_j,D^3u\cdot D^2u)\times \\ & \qquad\qquad\qquad\qquad\times u_{ij} |D^2u|^{\gamma-4}h_\tau(|D^2u|)G_{\epsilon}(|\nabla u|)^{\alpha+1}|\nabla u|^{-1} \psi^{2} \, dx\\
				&+  (p-2) \int_{\Omega} |\nabla u|^{p-4} (\nabla u_i,\nabla u)(\nabla u_j,D^2 u\nabla u)|\nabla u|^{-2} G_\varepsilon (|\nabla u|)^{\alpha} \times 
				\\ &  \qquad \qquad \qquad\times  ((\alpha+1)G'_\varepsilon(|\nabla u|)-G_\varepsilon(|\nabla u|)|\nabla u|^{-1}) G_\tau(|D^2u|)|D^2u|^{\gamma-2}u_{ij}\psi^{2} \, dx \\
				&+ 2(p-2) \int_{\Omega} |\nabla u|^{p-4} (\nabla u_i,\nabla u)(\nabla u_j,\nabla  \psi) G_{\epsilon}(|\nabla u|)^{\alpha+1}|\nabla u|^{-1} \times\\
				&\qquad \qquad \qquad \qquad \times G_\tau(|D^2u|)|D^2u|^{\gamma-2}u_{ij}\psi \, dx\\
				& - \int_{\Omega} f_{ij} G_{\epsilon}(|\nabla u|)^{\alpha+1}|\nabla u|^{-1} G_\tau(|D^2u|)|D^2u|^{\gamma-2} u_{ij} \psi^{2} \, dx =:I_1+\cdot\cdot\cdot+I_{25}=0
			\end{split}
		\end{equation*} 
		
		\noindent Since $G(t)\le t$ and $G'(t)\le 2$ for any $t\ge 0$, we deduce that 
		
		\begin{equation}\label{eq:sost277}
        \begin{split}
			& I_1+I_2+I_{13}+I_{14}= \int_{\Omega} |\nabla u|^{p-2} G_{\epsilon}(|\nabla u|)^{\alpha+1}|\nabla u|^{-1} G_\tau(|D^2u|)|D^2u|^{\gamma-2}|\nabla u_{ij}|^{2} \psi^{2} 
			\\  
            &+ \int_{\Omega} |\nabla u|^{p-2} G_{\epsilon}(|\nabla u|)^{\alpha+1}|\nabla u|^{-1} |D^2u|^{\gamma-4}h_\tau(|D^2u|) (\nabla u_{ij} u_{ij},D^3u\cdot D^2u) \psi^{2} \, dx
			\\ 
            &+(p-2) \int_{\Omega} |\nabla u|^{p-4} (\nabla u_{ij}, \nabla u)^{2}G_\tau(|D^2u|)|D^2u|^{\gamma-2}G_{\epsilon}(|\nabla u|)^{\alpha+1}|\nabla u|^{-1}\psi^{2}\,dx   \\  
            &+ (p-2) \int_{\Omega} |\nabla u|^{p-4} (\nabla u_{ij}u_{ij}, \nabla u)(D^3u\cdot D^2u, \nabla u)\times\\  &\qquad\qquad\qquad\qquad\times |D^2u|^{\gamma-4}h_\tau(|D^2 u|)G_{\epsilon}(|\nabla u|)^{\alpha+1}|\nabla u|^{-1}\psi^{2} \, dx
			\\  
       &\le C(\alpha,p,\gamma)\int_{\Omega} |\nabla u|^{p-3}G_\varepsilon(|\nabla u|)^{\alpha}  G_\tau(|D^2u|)|D^{2}u|^{\gamma} |D^3u|  \psi^{2} \,dx \\
      &+ C(p) \int_{\Omega} |\nabla u|^{p-3}G_\varepsilon(|\nabla u|)^{\alpha}|D^2u|^{\gamma-1}|h_\tau(|D^2u|)||D^3u\cdot D^2u|\psi^2 \,dx\\
			&+ C(p) \int_{\Omega} |\nabla u|^{p-4}G_\varepsilon(|\nabla u|)^{\alpha}|D^2u|^{\gamma-1}|h_\tau(|D^2u|)||(D^3u\cdot D^2u,\nabla u)|\psi^2 \,dx\\
			&+ C(p) \int_{\Omega} |\nabla u|^{p-2}G_\varepsilon(|\nabla u|)^{\alpha}   |D^3u| |\nabla \psi| G_\tau(|D^2u|)|D^2 u|^{\gamma-1} \psi\,dx \\
			&+ C(\alpha,p) \int_{\Omega\setminus Z_{u}} |\nabla u|^{p+\alpha-4} |D^{2}u|^{\gamma+3}  \psi^{2}\,dx \\
			&+ C(p) \int_{\Omega\setminus Z_{u}} |\nabla u|^{p+\alpha-3} |D^2 u|^{\gamma+2} |\nabla \psi| \psi \,dx\\
			& \qquad \qquad    +  \int_{\Omega}  |\nabla u|^{\alpha} |f_{ij}||D^2 u|^{\gamma} \psi^{2}\,dx,
        \end{split}
		\end{equation}
		where $C(p), C(\alpha,p)$ and $ C(\alpha,p,\gamma)$ are positive constants.
		
		Summing on $i,j=1,...,n$ we get 
		\begin{equation}\label{eq:sost2778}
		\begin{split}
			&\int_{\Omega} |\nabla u|^{p-2} G_{\epsilon}(|\nabla u|)^{\alpha+1}|\nabla u|^{-1} G_\tau(|D^2u|)|D^2u|^{\gamma-2}|D^3u|^{2} \psi^{2} 
			\\ 
            &+ \int_{\Omega} |\nabla u|^{p-2} G_{\epsilon}(|\nabla u|)^{\alpha+1}|\nabla u|^{-1} |D^2u|^{\gamma-4}h_\tau(|D^2u|))|D^3u\cdot D^2u|^2 \psi^{2} \, dx
			\\
   &+(p-2) \int_{\Omega} |\nabla u|^{p-4} \sum_{i,j=1}^n(\nabla u_{ij}, \nabla u)^{2}G_\tau(|D^2u|)|D^2u|^{\gamma-2}G_{\epsilon}(|\nabla u|)^{\alpha+1}|\nabla u|^{-1}\psi^{2}\,dx   \\  &+ (p-2) \int_{\Omega} |\nabla u|^{p-4} (D^3u\cdot D^2u, \nabla u)^2 |D^2u|^{\gamma-4}h_\tau(|D^2 u|)G_{\epsilon}(|\nabla u|)^{\alpha+1}|\nabla u|^{-1}\psi^{2} \, dx
			\\
        \end{split}
        \end{equation}
        \begin{equation*}
        \begin{split}
            &\le C(n,\alpha,p,\gamma)\int_{\Omega} |\nabla u|^{p-3}G_\varepsilon(|\nabla u|)^{\alpha}  G_\tau(|D^2u|)|D^{2}u|^{\gamma} |D^3u|  \psi^{2} \,dx \\ 
            &+ C(n,p) \int_{\Omega} |\nabla u|^{p-3}G_\varepsilon(|\nabla u|)^{\alpha}|D^2u|^{\gamma-1}|h_\tau(|D^2u|)||D^3u\cdot D^2u|\psi^2 \,dx\\
			&+ C(n,p) \int_{\Omega} |\nabla u|^{p-4}G_\varepsilon(|\nabla u|)^{\alpha}|D^2u|^{\gamma-1}|h_\tau(|D^2u|)||(D^3u\cdot D^2u,\nabla u)|\psi^2 \,dx\\
			&+ C(n,p) \int_{\Omega} |\nabla u|^{p-2}G_\varepsilon(|\nabla u|)^{\alpha}   |D^3u| |\nabla \psi| G_\tau(|D^2u|)|D^2 u|^{\gamma-1} \psi\,dx \\
   &+ C(n,\alpha,p) \int_{\Omega\setminus Z_{u}} |\nabla u|^{p+\alpha-4} |D^{2}u|^{\gamma+3}  \psi^{2}\,dx \\
			&+ C(n,p) \int_{\Omega\setminus Z_{u}} |\nabla u|^{p+\alpha-3} |D^2 u|^{\gamma+2} |\nabla \psi| \psi \,dx\\
			& \qquad \qquad    +  \, C(n)\int_{\Omega}  |\nabla u|^{\alpha} |D^2f||D^2 u|^{\gamma} \psi^{2}\,dx=:J_1+\cdot\cdot\cdot +J_7,
        \end{split}
		\end{equation*}
		where $C(n),C(n,p), C(n,\alpha,p)$ and $ C(n,\alpha,p,\gamma)$ are positive constants.
		
		Now we estimate the right-hand side of \eqref{eq:sost2778}. We estimate the term $J_1$. By weighted Young's inequality we get 
		
		\begin{equation}\label{stimaI_177}
			\begin{split}
				J_1&\le \theta\int_{\Omega} |\nabla u|^{p-2}G_{\epsilon}(|\nabla u|)^{\alpha+1}|\nabla u|^{-1} G_\tau(|D^2u|)|D^2u|^{\gamma-2}|D^3u|^2\psi^2\,dx\\ &+C(n,\gamma,p,\alpha,\theta)\int_{\Omega\setminus Z_{u}} |\nabla u|^{p+\alpha-4} |D^2 u|^{3+\gamma} \psi^2\,dx, 
			\end{split}
		\end{equation}
		where $C(n,\gamma,p,\alpha,\theta)$ is a positive constant. 
		
		For the term $J_2$, since $|h_\tau(t)|\le (2-\gamma)G_\tau (t)+G'_\tau (t)t$, for any $t\ge 0$, by \eqref{stimaI_177} and using a weighted Young's inequality we obtain 
		
		\begin{equation}\label{stimaI_277}
			\begin{split}
				J_2&\le C(n,p,\gamma)\int_{\Omega} |\nabla u|^{p-3}G_\varepsilon(|\nabla u|)^{\alpha}  G_\tau(|D^2u|)|D^{2}u|^{\gamma} |D^3u|  \psi^{2} \,dx \\
				&+C(n,p)\int_{\Omega} |\nabla u|^{p-3}G_\varepsilon(|\nabla u|)^{\alpha}|D^2u|^{\gamma-1}G'_\tau(|D^2u|)|D^2u|||D^3u\cdot D^2u|\psi^2\,dx\\
				&\le\theta\int_{\Omega} |\nabla u|^{p-2}G_{\epsilon}(|\nabla u|)^{\alpha+1}|\nabla u|^{-1} G_\tau(|D^2u|)|D^2u|^{\gamma-2}|D^3u|^2\psi^2\,dx\\
&+C(n,\gamma,p,\theta)\int_{\Omega\setminus Z_{u}} |\nabla u|^{p+\alpha-4} |D^2 u|^{3+\gamma} \psi^2\,dx \\
    \end{split}
    \end{equation}
    \begin{equation*}
    \begin{split}
            &+\theta\int_{\Omega} |\nabla u|^{p-2} G_{\epsilon}(|\nabla u|)^{\alpha+1}|\nabla u|^{-1} |D^2u|^{\gamma-4}G'_\tau(|D^2u|)|D^2u||D^3u\cdot D^2u|^2 \psi^{2} \, dx\\
&+C(n,p,\theta)\int_{\Omega\setminus Z_{u}} |\nabla u|^{p+\alpha-4} |D^2 u|^{3+\gamma} \psi^2\,dx \\
			\end{split}
		\end{equation*}
		where $C(n,\gamma,p,\theta)$ and $C(n,p,\theta)$ are positive constants. 
		We estimate the term $J_3$. By \eqref{stimaI_177} and by weighted Young inequality we get 
		
		\begin{equation}\label{StimaJ_3}
			\begin{split}
				J_3 &\le C(n,p,\gamma)\int_{\Omega} |\nabla u|^{p-3}G_\varepsilon(|\nabla u|)^{\alpha}  G_\tau(|D^2u|)|D^{2}u|^{\gamma} |D^3u|  \psi^{2} \,dx \\
				&+C(n,p)\int_{\Omega} |\nabla u|^{p-4}G_\varepsilon(|\nabla u|)^{\alpha}|D^2u|^{\gamma-1}G'_\tau(|D^2u|)|D^2u|||(D^3u\cdot D^2u,\nabla u)|\psi^2\,dx\\
				&\le\theta\int_{\Omega} |\nabla u|^{p-2}G_{\epsilon}(|\nabla u|)^{\alpha+1}|\nabla u|^{-1} G_\tau(|D^2u|)|D^2u|^{\gamma-2}|D^3u|^2\psi^2\,dx\\ 
				&+C(n,\gamma,p,\theta)\int_{\Omega\setminus Z_{u}} |\nabla u|^{p+\alpha-4} |D^2 u|^{3+\gamma} \psi^2\,dx \\
				&+\theta\int_{\Omega} |\nabla u|^{p-4} G_{\epsilon}(|\nabla u|)^{\alpha+1}|\nabla u|^{-1} |D^2u|^{\gamma-4}G'_\tau(|D^2u|)\times\\
    &\qquad\qquad\qquad\qquad\qquad\qquad\qquad\times|D^2u|(D^3u\cdot D^2u,\nabla u)^2 \psi^{2} \, dx\\
				&+C(n,p,\theta)\int_{\Omega\setminus Z_{u}} |\nabla u|^{p+\alpha-4} |D^2 u|^{3+\gamma} \psi^2\,dx ,
			\end{split}
		\end{equation}
		where $C(n,\gamma,p,\theta)$ and C$(n,p,\theta)$ are positive constant. For the term $J_4$, by weighted Young's inequality we have 
		
		\begin{equation}\label{stimaJ_4}
			\begin{split}
				J_4&\le \theta  \int_{\Omega} |\nabla u|^{p-2}G_{\epsilon}(|\nabla u|)^{\alpha+1}|\nabla u|^{-1} G_\tau(|D^2u|)|D^2u|^{\gamma-2}|D^3u|^2\psi^2\,dx\\ 
				&+ C(n,p,\theta)\int_\Omega |\nabla u|^{p-2+\alpha}|D^2u|^{\gamma+1}|\nabla \psi|^2 \,dx,
			\end{split}
		\end{equation}
		where $C(n,p,\theta)$ is a positive constant. 
		Now we set $h_{\tau,\theta}(t):=(\gamma-2)G_\tau(t)+(1-\theta)G'_\tau(t)t$, for any $t\ge 0$. By \eqref{stimaI_177}, \eqref{stimaI_277}, \eqref{StimaJ_3} and \eqref{stimaJ_4} we get
		
		\begin{equation*}
        \begin{split}
			&(1-4\theta)\int_{\Omega} |\nabla u|^{p-2} G_{\epsilon}(|\nabla u|)^{\alpha+1}|\nabla u|^{-1} G_\tau(|D^2u|)|D^2u|^{\gamma-2}|D^3u|^{2} \psi^{2} \,dx
			\\
   &+ \int_{\Omega} |\nabla u|^{p-2} G_{\epsilon}(|\nabla u|)^{\alpha+1}|\nabla u|^{-1} |D^2u|^{\gamma-4}h_{\tau,\theta}(|D^2u|)|D^3u\cdot D^2u|^2 \psi^{2} \, dx
			\\ 
   \end{split}
   \end{equation*}
   \begin{equation}\label{eq:sost27789}
       \begin{split}
            &+(p-2) \int_{\Omega} |\nabla u|^{p-4} \sum_{i=1}^n(\nabla u_{ij}, \nabla u)^{2}G_\tau(|D^2u|)|D^2u|^{\gamma-2}G_{\epsilon}(|\nabla u|)^{\alpha+1}|\nabla u|^{-1}\psi^{2}\,dx   \\ 
            &+ (p-2) \int_{\Omega} |\nabla u|^{p-4} (D^3u\cdot D^2u, \nabla u)^2 |D^2u|^{\gamma-4}h_{\tau,\theta}(|D^2 u|)G_{\epsilon}(|\nabla u|)^{\alpha+1}|\nabla u|^{-1}\psi^{2} \, dx
			\\  
   &\le  C(n,\alpha,p,\gamma,\theta) \int_{\Omega\setminus Z_{u}} |\nabla u|^{p+\alpha-4} |D^{2}u|^{\gamma+3}  \psi^{2}\,dx \\
			&+ C(n,p) \int_{\Omega\setminus Z_{u}} |\nabla u|^{p+\alpha-3} |D^2 u|^{\gamma+2} |\nabla \psi| \psi \,dx\\
	 &+ C(n,p,\theta)\int_\Omega |\nabla u|^{p-2+\alpha}|D^2u|^{\gamma+1}|\nabla \psi|^2 \,dx\\
			& \qquad \qquad    +  C(n)\int_{\Omega}  |\nabla u|^{\alpha} |D^2f||D^2 u|^{\gamma} \psi^{2}\,dx,
        \end{split}
		\end{equation}
		where $C(n,\alpha,p,\gamma,\theta)$ is a positive constant.
		
		Now we split the domain $\Omega$ into three subdomains as follows $$\Omega=\{|D^2u|\le \tau\}\cup \{ \tau<|D^2u|< 2\tau\} \cup \{|D^2u|\ge 2\tau\}=:\Omega_0\cup\Omega_1\cup\Omega_2.$$
		Now we estimate the left-hand side of \eqref{eq:sost27789}. Recalling the definition $G_\tau$ (see \eqref{gepsilon}), we get 
		
		\begin{equation}\label{basta}
			\begin{split}
				& \tilde J_1+\cdot\cdot\cdot+\tilde J_8 \\
				:=& (1-4\theta)\int_{\Omega_1} |\nabla u|^{p-2} G_{\epsilon}(|\nabla u|)^{\alpha+1}|\nabla u|^{-1} G_\tau(|D^2u|)|D^2u|^{\gamma-2}|D^3u|^{2} \psi^{2} \,dx
				\\ +& \int_{\Omega_1} |\nabla u|^{p-2} G_{\epsilon}(|\nabla u|)^{\alpha+1}|\nabla u|^{-1} |D^2u|^{\gamma-4}h_{\tau,\theta}(|D^2u|)|D^3u\cdot D^2u|^2 \psi^{2} \, dx
				\\ 
                +&(p-2) \int_{\Omega_1} |\nabla u|^{p-4} \sum_{i,j=1}^n(\nabla u_{ij}, \nabla u)^{2}G_\tau(|D^2u|)|D^2u|^{\gamma-2}G_{\epsilon}(|\nabla u|)^{\alpha+1}|\nabla u|^{-1}\psi^{2}\,dx   \\
                +& (p-2) \int_{\Omega_1} |\nabla u|^{p-4} (D^3u\cdot D^2u, \nabla u)^2 |D^2u|^{\gamma-4}h_{\tau,\theta}(|D^2 u|)\times\\    &\qquad\qquad\qquad\qquad\qquad\qquad\qquad\qquad\times G_{\epsilon}(|\nabla u|)^{\alpha+1}|\nabla u|^{-1}\psi^{2} \, dx\\
             +& (1-4\theta)\int_{\Omega_2} |\nabla u|^{p-2} G_{\epsilon}(|\nabla u|)^{\alpha+1}|\nabla u|^{-1} G_\tau(|D^2u|)|D^2u|^{\gamma-2}|D^3u|^{2} \psi^{2} \,dx
				\\
    +& \int_{\Omega_2} |\nabla u|^{p-2} G_{\epsilon}(|\nabla u|)^{\alpha+1}|\nabla u|^{-1} |D^2u|^{\gamma-4}h_{\tau,\theta}(|D^2u|)|D^3u\cdot D^2u|^2 \psi^{2} \, dx
				\\ 
                \end{split}
                \end{equation}
                \begin{equation*}  
                \begin{split}     
    +&(p-2) \int_{\Omega_2} |\nabla u|^{p-4} \sum_{i,j=1}^n(\nabla u_{ij}, \nabla u)^{2}G_\tau(|D^2u|)|D^2u|^{\gamma-2}G_{\epsilon}(|\nabla u|)^{\alpha+1}|\nabla u|^{-1}\psi^{2}\,dx   \\ +& (p-2) \int_{\Omega_2} |\nabla u|^{p-4} (D^3u\cdot D^2u, \nabla u)^2 |D^2u|^{\gamma-4}h_{\tau,\theta}(|D^2 u|)\times\\    &\qquad\qquad\qquad\qquad\qquad\qquad\qquad\qquad\times G_{\epsilon}(|\nabla u|)^{\alpha+1}|\nabla u|^{-1}\psi^{2} \, dx.
			\end{split}
		\end{equation*}

		We consider the case $p<2$. For $\max\{(p-2)/(p-1),2-p\}<\gamma<1$ fixed and $\theta>0$ sufficiently small, we note that $h_{\tau,\theta}>0$ in $\Omega_1$ and $h_{\tau,\theta}(t)=(\gamma-1-\theta)t=(\gamma-1-\theta)G_\tau(t)<0$ in $\Omega_2$. Using Cauchy-Schwarz inequality we have 
		\begin{equation}\label{abc}
			\begin{split}
				&  (p-2) \int_{\Omega_1} |\nabla u|^{p-4} (D^3u\cdot D^2u, \nabla u)^2 |D^2u|^{\gamma-4}h_{\tau,\theta}(|D^2 u|)G_{\epsilon}(|\nabla u|)^{\alpha+1}|\nabla u|^{-1}\psi^{2} \, dx 
				\\ & \ge (p-2)\int_{\Omega_1} |\nabla u|^{p-2} G_{\epsilon}(|\nabla u|)^{\alpha+1}|\nabla u|^{-1} |D^2u|^{\gamma-4}h_{\tau,\theta}(|D^2u|)|D^3u\cdot D^2u|^2 \psi^{2} \, dx.
			\end{split}
		\end{equation}
		
		By \eqref{abc} and Cauchy-Schwarz inequality we have 
		
		\begin{equation}\label{pmin2}
			\begin{split}
				& \tilde J_1+\cdot\cdot\cdot+\tilde J_8 
				\\ \ge &  (p-1-4\theta)\int_{\Omega_1} |\nabla u|^{p-2} G_{\epsilon}(|\nabla u|)^{\alpha+1}|\nabla u|^{-1} G_\tau(|D^2u|)|D^2u|^{\gamma-2}|D^3u|^{2} \psi^{2} \,dx
				\\ +& (p-1)\int_{\Omega_1} |\nabla u|^{p-2} G_{\epsilon}(|\nabla u|)^{\alpha+1}|\nabla u|^{-1} |D^2u|^{\gamma-4}h_{\tau,\theta}(|D^2u|)|D^3u\cdot D^2u|^2 \psi^{2} \, dx
				\\ +&  (\gamma+p-2-5\theta)\int_{\Omega_2} |\nabla u|^{p-2} G_{\epsilon}(|\nabla u|)^{\alpha+1}|\nabla u|^{-1} G_\tau(|D^2u|)|D^2u|^{\gamma-2}|D^3u|^{2} \psi^{2} \,dx
				\\ \ge& (\gamma+p-2-5\theta)\int_{\Omega_2} |\nabla u|^{p-2} G_{\epsilon}(|\nabla u|)^{\alpha+1}|\nabla u|^{-1}\times\\
    &\qquad\qquad\qquad\qquad\qquad\qquad\qquad\times G_\tau(|D^2u|)|D^2u|^{\gamma-2}|D^3u|^{2} \psi^{2} \,dx,
			\end{split}
		\end{equation}
		where in the last inequality we have used the fact that $h_{\tau,\theta}>0$ in $\Omega_1$, for $\theta$ sufficiently small.

		In the case $p\ge 2$, by Cauchy-Schwarz inequality we have

		\begin{equation}\label{pmax2}
			\begin{split}
				& \tilde J_1+\cdot\cdot\cdot+\tilde J_8 
				\\ \ge &  (1-4\theta)\int_{\Omega_1} |\nabla u|^{p-2} G_{\epsilon}(|\nabla u|)^{\alpha+1}|\nabla u|^{-1} G_\tau(|D^2u|)|D^2u|^{\gamma-2}|D^3u|^{2} \psi^{2} \,dx
				\\ +&  (\gamma+(p-2)(\gamma-1-\theta)-5\theta)\int_{\Omega_2} |\nabla u|^{p-2} G_{\epsilon}(|\nabla u|)^{\alpha+1}|\nabla u|^{-1} \times \\ & \qquad\qquad\qquad\qquad\qquad\qquad\qquad\times G_\tau(|D^2u|)|D^2u|^{\gamma-2}|D^3u|^{2} \psi^{2} \,dx \\
				\ge &(\gamma+(p-2)(\gamma-1-\theta)-5\theta)\int_{\Omega_2} |\nabla u|^{p-2} G_{\epsilon}(|\nabla u|)^{\alpha+1}|\nabla u|^{-1} \times \\ & \qquad\qquad\qquad\qquad\qquad\qquad\qquad\times G_\tau(|D^2u|)|D^2u|^{\gamma-2}|D^3u|^{2} \psi^{2} \,dx.
			\end{split}
		\end{equation}
		
		Using \eqref{pmin2} and \eqref{pmax2} in \eqref{eq:sost27789}, and since  $\gamma>\max\{(p-2)/(p-1),2-p\}$,  for $\theta>0$ sufficiently small, there exists a positive constant $C(\gamma,p,\theta)$ such that

		\begin{eqnarray} \label{eq:sost277890}
			\nonumber &&  C(\gamma,p,\theta)\int_{\Omega} |\nabla u|^{p-2} G_{\epsilon}(|\nabla u|)^{\alpha+1}|\nabla u|^{-1}  G_\tau(|D^2u|)|D^2u|^{\gamma-2}|D^3u|^{2} \psi^{2} \,dx
			\\ \nonumber &\le&  C(n,\alpha,p,\gamma,\theta) \int_{\Omega\setminus Z_{u}} |\nabla u|^{p+\alpha-4} |D^{2}u|^{\gamma+3}  \psi^{2}\,dx \\
			\nonumber &+& C(n,p) \int_{\Omega\setminus Z_{u}} |\nabla u|^{p+\alpha-3} |D^2 u|^{\gamma+2} |\nabla \psi| \psi \,dx\\
			\nonumber &+&C(n,p,\theta)\int_\Omega |\nabla u|^{p-2+\alpha}|D^2u|^{\gamma+1}|\nabla \psi|^2 \,dx\\
			&& \qquad \qquad    +  C(n)\int_{\Omega}  |\nabla u|^{\alpha} |D^2f||D^2 u|^{\gamma} \psi^{2}\,dx,
		\end{eqnarray}
		
		By Fatou Lemma, for $\tau\rightarrow 0$ we obtain 
		
		\begin{eqnarray} \label{eq:sost2778901}
			\nonumber &&  C(\gamma,p,\theta)\int_{\Omega\setminus \{|D^2u|=0\}} |\nabla u|^{p-2} G_{\epsilon}(|\nabla u|)^{\alpha+1}|\nabla u|^{-1}  |D^2u|^{\gamma-1}|D^3u|^{2} \psi^{2} \,dx
			\\ \nonumber &\le&  C(n,\alpha,p,\gamma,\theta) \int_{\Omega\setminus Z_{u}} |\nabla u|^{p+\alpha-4} |D^{2}u|^{\gamma+3}  \psi^{2}\,dx \\
			\nonumber &+& C(n,p) \int_{\Omega\setminus Z_{u}} |\nabla u|^{p+\alpha-3} |D^2 u|^{\gamma+2} |\nabla \psi| \psi \,dx\\
			\nonumber &+&C(n,p,\theta)\int_\Omega |\nabla u|^{p-2+\alpha}|D^2u|^{\gamma+1}|\nabla \psi|^2 \,dx\\
			&& \qquad \qquad    +  C(n)\int_{\Omega}  |\nabla u|^{\alpha} |D^2f||D^2 u|^{\gamma} \psi^{2}\,dx.
		\end{eqnarray}
		
		Proceeding as in the case $\gamma \ge 1$, we get the thesis.
		
	\end{proof}

	Now we state a result which will be useful to us later
	
	\begin{prop}\label{soluzioneregolarizzata}
		Let $u_\varepsilon\in C^{1,\beta}_{loc}(\Omega)$ be a weak solution of the regularized problem \eqref{eq:problregol}. Let $f,q,p,\alpha$ obey the assumptions of Theorem \ref{derivate terze}. Then for $\gamma\ge 1$, and for any $\tilde \Omega \subset\subset\Omega$, we have that
		\begin{equation}\label{derivateterze55}
			\int_{\tilde \Omega } (\varepsilon+|\nabla u_\varepsilon|^2)^{\frac{p-2+\alpha}{2}} |D^2u_\varepsilon|^{\gamma-1}|D^3u_\varepsilon|^{2} \, dx \le C,
		\end{equation}
		for all $\varepsilon\in (0,1)$, where $C$ is a positive constant not depending on $\varepsilon$.
	\end{prop}
	
	\begin{proof}
		The proof is similar to that of Theorem \ref{derivate terze}. However, we will give a sketch of the proof. First we consider the second linearized equation of \eqref{eq:problregol}
		at any fixed solution $u_\varepsilon$, which satisfies 
		\begin{equation}\label{equazione debole3}
			\int_\Omega (\varepsilon+|\nabla u_\varepsilon|^2)^{\frac{p-2}{2}}(\nabla u_\varepsilon,\nabla \psi)\,dx=\int_\Omega f \psi \,dx \quad \forall \psi\in C^{\infty}_c(\Omega).
		\end{equation}
		For $\varphi \in C^{\infty}_c(\Omega )$, taking $\psi :=\varphi_{ij}$ in \eqref{equazione debole3}, we get 
		
		\begin{eqnarray} \label{eq:lin secregolarizzato}
			\nonumber && \int_{\Omega} (\varepsilon+|\nabla u_\varepsilon|^2)^{\frac{p-2}{2}} (\nabla u_{\varepsilon,ij},\nabla \varphi) \\
			\nonumber &+& (p-2) \int_{\Omega} (\varepsilon+|\nabla u_\varepsilon|^2)^{\frac{p-4}{2}} (\nabla u_\varepsilon,\nabla u_{\varepsilon,j})(\nabla u_{\varepsilon,i}, \nabla \varphi) \\
			\nonumber &+& (p-2)(p-4) \int_{\Omega} (\varepsilon+|\nabla u_\varepsilon|^2)^{\frac{p-6}{2}} (\nabla u_\varepsilon,\nabla u_{\varepsilon,j})(\nabla u_{\varepsilon,i},\nabla u_\varepsilon)(\nabla u_\varepsilon,\nabla \varphi) \\
			\nonumber &+& (p-2) \int_{\Omega} (\varepsilon+|\nabla u_\varepsilon|^2)^{\frac{p-4}{2}} (\nabla u_{\varepsilon,ij},\nabla u_\varepsilon)(\nabla u_\varepsilon,\nabla \varphi)\\
			\nonumber &+& (p-2) \int_{\Omega} (\varepsilon+|\nabla u_\varepsilon|^2)^{\frac{p-4}{2}} (\nabla u_{\varepsilon,i},\nabla u_{\varepsilon,j})(\nabla u_\varepsilon,\nabla \varphi) \\
			\nonumber &+& (p-2) \int_{\Omega} (\varepsilon+|\nabla u_\varepsilon|^2)^{\frac{p-4}{2}} (\nabla u_{\varepsilon,i},\nabla u_\varepsilon)(\nabla u_{\varepsilon,j},\nabla \varphi) = \int_{\Omega} f_{ij} \, \varphi. \\
		\end{eqnarray}

		Let us define 
		\begin{equation}\label{funzionetest3}
			\varphi := (\varepsilon+|\nabla u_\varepsilon|^2)^{\frac{\alpha}{2}} \, |D^2u_{\varepsilon}|^{\gamma-1}u_{\varepsilon,ij} \, \psi^{2},
		\end{equation}
		where $\psi$ is defined in \eqref{eq:psi}.
		Substituting \eqref{funzionetest3} in \eqref{eq:lin secregolarizzato}, following the proof of Theorem \ref{derivate terze}, we can deduce 
		\begin{equation}\label{eq:sost94a}
			\begin{split}
				C(p)&\int_{\Omega} (\varepsilon+|\nabla u_\varepsilon|^2)^{\frac{p-2+\alpha}{2}}  |D^2u_\varepsilon|^{\gamma-1}| D^3 u_\varepsilon|^{2} \psi^{2} \,dx 
				\\ &\le C(n,\alpha,p,\gamma)\int_{\Omega} (\varepsilon+|\nabla u_\varepsilon|^2)^{\frac{p-3+\alpha}{2}}  |D^{2}u_\varepsilon|^{\gamma+1} |D^3u_\varepsilon|  \psi^{2} \,dx \\
				&+ C(n,p)\int_{\Omega} (\varepsilon+|\nabla u_\varepsilon|^2)^{\frac{p-2+\alpha}{2}}  |\nabla \psi||D^2u_\varepsilon|^{\gamma} |D^3 u_\varepsilon| \psi\,dx \\
				&+C(n,\alpha,p) \int_{\Omega} (\varepsilon+|\nabla u_\varepsilon|^2)^{\frac{p+\alpha-4}{2}} |D^{2}u_\varepsilon|^{\gamma+3}  \psi^{2}\,dx \\
				&+C(n,p) \int_{\Omega} (\varepsilon+|\nabla u_\varepsilon|^2)^{\frac{p-3+\alpha}{2}} |D^2 u_\varepsilon|^{\gamma+2} |\nabla \psi| \psi \,dx\\
				&+C(n)\int_{\Omega}  (\varepsilon+|\nabla u_\varepsilon|^2)^{\frac{\alpha}{2}} |D^2f||D^2 u_\varepsilon|^{\gamma} \psi^{2}\,dx,
			\end{split}
		\end{equation}
		
		where $C(n,\alpha,p,\gamma)$, $C(n,\alpha,p)$, $C(n,p)$, $C(p)$ and $C(n)$ are positive constants.
		
		Proceeding as in the proof of Theorem \ref{derivate terze}, we get the thesis.

	\end{proof}

	In the case in which the function $f$ does not have a sign we can state the following 
	
	\begin{thm}\label{derivateterze2}
		Let $\Omega\subset \mathbb{R}^n$ be a domain and let $u\in C^{1,\beta}_{loc}(\Omega)$ be a weak solution of \eqref{eq:problema}. 
		
		For $\gamma> \max \{(p-2)/(p-1),2-p\}$, we set 
		\begin{equation} \label{definizionediq_n2}
			\begin{cases}
				q_N = 2(q_{N-1}-1) \\
				q_0 = 3+\gamma.\\
			\end{cases}
		\end{equation}
		Let $q\ge q_0$, with $q_N\le  q< q_{N+1}$ for some $N \in \mathbb{N}_0$ and $p$ be such that: 
		\begin{equation}\label{evvai3}
			2-\frac{1}{C(n,q)}< p <\min \left\{2+\frac{1}{q-1} ,2+\frac{1}{C(n,q)}\right\},
		\end{equation}
		where $C(n,q)$ is given by \eqref{calderonintoduzione}.  Assume \begin{equation}\label{valorealpha2}
			\alpha> \alpha(p,N):=\frac{3-p}{2^N}+1,
		\end{equation}
        if $N=0$, $\alpha\ge 4-p.$
		Let $f(x)\in W^{2,q/(q-\gamma)}_{loc}(\Omega)\cap C^{1,\beta'}_{loc}(\Omega)$. Then, for any $\tilde \Omega \subset \subset \Omega$, for $\gamma\ge 1$, we have that 
		\begin{equation}\label{stima derivata terza3}
			\int_{\tilde \Omega \setminus Z_u} |\nabla u|^{p-2+\alpha} |D^2u|^{\gamma-1}|D^3u|^{2} \, dx \le C,
		\end{equation}
		where $C =C(\gamma,\alpha,p,N,q,n,f,\tilde \Omega,\|\nabla u\|_{L^{\infty}_{loc}(\Omega)})$.
		
		If $ \max \{(p-2)/(p-1),2-p\}<\gamma<1$, for any $\tilde \Omega \subset \subset \Omega$, we deduce that
		\begin{equation}\label{stima derivata terza333}
			\int_{\tilde \Omega \setminus (Z_u\cup \{|D^2u|=0\})} |\nabla u|^{p-2+\alpha} |D^2u|^{\gamma-1}|D^3u|^{2} \, dx \le C,
		\end{equation}
		where $C =C(\gamma,\alpha,p,q,n,f,\tilde \Omega,\|\nabla u\|_{L^{\infty}_{loc}(\Omega)},N)$.
		
	\end{thm}
	
	\begin{proof}
		We consider the case $\gamma>1$, the case $\max\{(p-2)/(p-1),2-p\}<\gamma<1$ is similar. The proof is similar to that of the Theorem \ref{derivate terze}. The key distinction lies in the non-application of Theorem \ref{inversodelpeso}. However, we will give a sketch of the proof for the reader's convenience. 
		
		Let us define 
\begin{equation}\label{funzionetest2}
			\varphi := \frac{G_{\epsilon}(|\nabla u|)^{\alpha+1}}{|\nabla u|} \, |D^2u|^{\gamma-1}u_{ij} \, \psi^{2},
		\end{equation}
		where $G_{\epsilon}$ and $\psi$ are defined in \eqref{gepsilon} and \eqref{eq:psi}.
		Substituting \eqref{funzionetest2} in \eqref{eq:lin sec}, following the proof of the Theorem \ref{derivate terze}, we obtain \eqref{eq:sost4}, that is
\begin{equation}\label{eq:sost94}
			\begin{split}
				C(p)&\int_{\Omega} |\nabla u|^{p-2} G_{\epsilon}(|\nabla u|)^{\alpha+1}|\nabla u|^{-1} |D^2u|^{\gamma-1}| D^3 u|^{2} \psi^{2} \,dx 
				\\ &\le C(n,\alpha,p,\gamma)\int_{\Omega} |\nabla u|^{p-3}G_{\epsilon}(|\nabla u|)^{\alpha}  |D^{2}u|^{\gamma+1} |D^3u|  \psi^{2} \,dx \\
				&+ C(n,p)\int_{\Omega} |\nabla u|^{p-2}G_{\epsilon}(|\nabla u|)^{\alpha}  |\nabla \psi||D^2u|^{\gamma} |D^3 u| \psi\,dx \\
				&+C(n,\alpha,p) \int_{\Omega} |\nabla u|^{p+\alpha-4} |D^{2}u|^{\gamma+3}  \psi^{2}\,dx \\
				&+C(n,p) \int_{\Omega} |\nabla u|^{p+\alpha-3} |D^2 u|^{\gamma+2} |\nabla \psi| \psi \,dx\\
				&+C(n)\int_{\Omega}  |\nabla u|^{\alpha} |D^2f||D^2 u|^{\gamma} \psi^{2}\,dx=:I_1+I_2+I_3+I_4+I_5,
			\end{split}
		\end{equation}
		
		where $C(n,\alpha,p,\gamma)$, $C(n,\alpha,p)$, $C(n,p)$, $C(p)$ and $C(n)$ are positive constants.
		Now we estimate the term $I_3$. Using a standard Young inequality we obtain 
		\begin{equation}
			\begin{split}
				I_3 & \le C(n,\alpha,p)\int_{B_{2R}} |\nabla u|^{p+\alpha-4} |D^{2}u|^{3+\gamma}\,dx   
				\\ &=C(n,\alpha,p)\int_{B_{2R}} |\nabla u|^{p+\alpha-4} |\nabla u|^{\frac{p-2-\beta}{2}} |\nabla u|^{\frac{2-p+\beta}{2}} |D^{2}u| |D^{2}u|^{2+\gamma}\, dx  
				\\ &\le \frac{C(n,\alpha,p)}{2} \int_{B_{2R}} |\nabla u|^{p-2-\beta} |D^{2}u|^{2} \,dx \\
				&+ \frac{C(n,\alpha,p)}{2} \int_{B_{2R}} |\nabla u|^{2(p-4+\alpha) +2-p+\beta} |D^{2}u|^{4+2\gamma} \,dx 
				\\ &\le C(n,\alpha,p,R,\beta,f, \|\nabla u\|_{L^{\infty}_{loc}(\Omega)}) \\
				&+ C(n,\alpha,p) \int_{B_{2R}} |\nabla u|^{2(p-4+\alpha) +2-p+\beta} |D^{2}u|^{4+2\gamma} \,dx,   
			\end{split}
		\end{equation}

		\noindent where in the last inequality we have used Theorem \ref{teosecond} and \newline $C(n,\alpha,p,R,\beta,f, \|\nabla u\|_{L^{\infty}_{loc}(\Omega)})$ is a positive constant. Iterating this procedure $N$-times, we get 
		\begin{equation}
			\begin{split}
				I_3& \le NC(n,\alpha,p,R,\beta,f, \|\nabla u\|_{L^{\infty}_{loc}(\Omega)})
				\\ &+C(n,\alpha,p)\int_{B_{2R}} |\nabla u|^{2^N(p-4+\alpha) +\sum_{k=1}^{N}2^{k-1}(2-p+\beta)} |D^{2}u|^{q_N} \,dx, 
			\end{split}
		\end{equation}
		
		where $q_N$ is given by 
		
		\begin{equation} \label{definizionediq_n3}
			\begin{cases}
				q_N = 2(q_{N-1}-1) \\
				q_0 = 3+\gamma.\\
			\end{cases}
		\end{equation}
		
		For $\beta\approx 1$, we note that $\alpha > (3-p)/2^N+1$ yields
		\begin{equation}\label{newvalore}
			2^N(p-4+\alpha) +\sum_{k=1}^{N}2^{k-1}(2-p+\beta)\ge 0.
		\end{equation}
        We note that if $N=0$, then we have $\alpha\ge 4-p.$
        
		By \eqref{newvalore}, using Theorem \ref{teoremacalderon} we get
		
		\begin{equation}\label{key}
			\begin{split}
				I_3 &\le  NC(n,\alpha,p,R,\beta,f, \|\nabla u\|_{L^{\infty}_{loc}(\Omega)})
				\\ & +C(n,\alpha,p,\|\nabla u\|_{L^{\infty}_{loc}(\Omega)})\int_B |D^2u|^{q_N}\,dx \\& \le C(n,q,\alpha,p,f,q,B,\|\nabla u\|_{L^{\infty}_{loc}(\Omega)},N), 
			\end{split}
		\end{equation}
		
		where $B:=B_{2R}(x_0)\subset\subset\Omega$, with $x_0\in \Omega$, and $C(n,q,\alpha,p,f,q,B,\|\nabla u\|_{L^{\infty}_{loc}(\Omega)},N)$ is a positive constant. Using \eqref{eq:sost94} and proceeding as in the proof of Theorem \ref{derivate terze}, we obtain the thesis.

	\end{proof}
	
	\
	
	\begin{rem}\label{remarkino}
		Under the assumptions of Theorem \ref{derivateterze2}, for a weak solution $u_\varepsilon$ of the regularized problem \eqref{eq:problregol}, the following inequality, for $\gamma\ge 1$, holds
		\begin{equation}\label{derivateterze5}
			\int_{\tilde \Omega } (\varepsilon+|\nabla u_\varepsilon|^2)^{\frac{p-2+\alpha}{2}} |D^2u_\varepsilon|^{\gamma-1}|D^3u_\varepsilon|^{2} \, dx \le C,
		\end{equation}
		for all $\varepsilon\in [0,1)$, where $C$ is a positive constant not depending on $\varepsilon.$ The case $\varepsilon >0$ follows by the same arguments used in the proof of  Theorem \ref{derivateterze2} and Proposition \ref{soluzioneregolarizzata}.
		
	\end{rem}
	\
	
	\

	Now we are ready to prove Theorem \ref{chestress}
	
	\begin{proof}[Proof of Theorem \ref{chestress}]
		First we suppose that $f \ge \tau > 0$ in $\Omega.$ For $\varepsilon>0$ sufficiently small we set \begin{equation}\label{eq:veps}
			V_{\epsilon,i} := h_{\epsilon}(|\nabla u|) |\nabla u|^{k-2} u_{i},
		\end{equation}
		where $k>0,$  $h_{\epsilon}(t):= \epsilon \, h(\frac{t}{\epsilon})$ and $h \in C^{2}(\R^{+})$ is such that 
		\begin{equation}
			h(t) := 
			\begin{cases}
				0               & \text{if} \quad t \in [0,1]\\
				t               & \text{if} \quad t \in [2, \infty),\\
			\end{cases}
		\end{equation}
		with $h(t)\le t$ for $t \in [1,2]$. We note that there exists a positive constant $\tilde C$ such that $h'_{\epsilon}(t) \le \Tilde{C}$ and $h''_{\epsilon}(t) \le \frac{\Tilde{C}}{\epsilon}$ for any $t\ge 0$.

		Then, we have
		\begin{equation}
			\begin{split}
				\frac{\partial V_{\epsilon,i}}{\partial x_j} &= h_{\epsilon}(|\nabla u|) |\nabla u|^{k -2}    u_{ij} + (k-2) h_{\epsilon}(|\nabla u|) |\nabla u|^{k -4} \langle \nabla u_{j}, \nabla u \rangle u_i \\
				&+ h^{'}_{\epsilon}(|\nabla u|) |\nabla u|^{k-3} \langle \nabla u_{j}, \nabla u \rangle u_i.
			\end{split}
		\end{equation}
		We set $B := B_{R}(x_0) \subset \subset \Omega,$ with $x_0 \in \Omega.$ Moreover, for $\tilde \delta>0$, yet to be determined, sufficiently small, we set $r:=1+\tilde\delta$. Since $h_{\epsilon}(t) \le t$ and $h'_{\epsilon}(t) \le \Tilde{C},$ using a Holder inequality with exponents $(\frac{2}{r},\frac{2}{2-r})$ we deduce that
		\begin{equation}\label{eq:veps1}
			\begin{split}
				\int_{B} \left| \frac{\partial V_{\epsilon,i}}{\partial x_j} \right|^{r} \ dx 
				&\le C(k) \int_{B\setminus Z_u} |\nabla u|^{r(k-1)} |D^2u|^{r} \ dx \\
				&\le C(k) |B|^{\frac{2-r}{2}} \left( \int_{B\setminus Z_u} |\nabla u|^{2(k-1)} |D^2u|^{2} \ dx \right)^{\frac{r}{2}} , \\
			\end{split}
		\end{equation}
		where $C(k)$ is a positive constant.
		Now we note that it is possible to find $\beta$ close to $1$ such that
		\begin{equation}
			2(k-1) \ge p-2-\beta \iff k >\frac{p-1}{2}.
		\end{equation}

		So using Theorem \ref{teosecond}  we have that
		\begin{equation}
			\int_{B} \left| \frac{\partial V_{\epsilon,i}}{\partial x_j} \right|^r \ dx \le C,
		\end{equation}
		with $C$ not depending on $\epsilon.$
		Moreover
		\begin{equation}
			\begin{split}
				\frac{\partial V_{\epsilon,i}}{\partial x_j \partial x_l} 
				&= h^{'}_{\epsilon}(|\nabla u|) \langle \nabla u_l, \nabla u \rangle |\nabla u|^{k-3} u_{ij} 
				+ (k-2) h_{\epsilon}(|\nabla u|) |\nabla u|^{k-4} \langle \nabla u_l, \nabla u \rangle u_{ij} \\
				&+h_{\epsilon}(|\nabla u|) |\nabla u|^{k-2} u_{ijl} + (k-2) h^{'}_{\epsilon}(|\nabla u|) \langle \nabla u_l, \nabla u \rangle |\nabla u|^{k-5} \langle \nabla u_j, \nabla u \rangle u_i \\
				&+ (k-2)(k-4) h_{\epsilon}(|\nabla u|) |\nabla u|^{k-6} \langle \nabla u_l, \nabla u \rangle \langle \nabla u_j, \nabla u \rangle u_i\\
				&+ (k-2) h_{\epsilon}(|\nabla u|)  |\nabla u|^{k-4} \left( \langle \nabla u_{jl}, \nabla u \rangle +\langle \nabla u_j, \nabla u_l \rangle \right) u_i \\
				&+ (k-2) h_{\epsilon}(|\nabla u|) |\nabla u|^{k-4} \langle \nabla u_j, \nabla u \rangle u_{il}\\
				&+ h^{''}_{\epsilon}(|\nabla u|) \langle \nabla u_j, \nabla u \rangle \langle \nabla u_l, \nabla u \rangle |\nabla u|^{k-4} u_i\\
				&+ h^{'}_{\epsilon}(|\nabla u|) \left( \langle \nabla u_{jl}, \nabla u \rangle + \langle \nabla u_j, \nabla u_l \rangle \right) |\nabla u|^{k-3} u_i \\
				&+ (k-3) h^{'}_{\epsilon}(|\nabla u|) \langle \nabla u_{j}, \nabla u \rangle \langle \nabla u_{l}, \nabla u \rangle |\nabla u|^{k-5} u_i\\
				&+ h^{'}_{\epsilon}(|\nabla u|) \langle \nabla u_{j}, \nabla u \rangle |\nabla u|^{k-3} u_{il}.
			\end{split}
		\end{equation}
		
		Since $h_{\epsilon}(t) \le t$, $h'_{\epsilon}(t) \le \Tilde{C}$ and $h''_{\epsilon}(t) \le \Tilde{C}/t,$  we deduce that
		\begin{equation}\label{eq:vvvv}
			\begin{split}
				\int_{B} &\left| \frac{\partial V_{\epsilon,i}}{\partial x_j \partial x_l}  \right|^{r} \ dx \\
				&\le C(k) \int_{B\setminus Z_u} |\nabla u|^{r(k-2)} |D^2u|^{2r} \ dx 
				+ C(k) \int_{B\setminus Z_u} |\nabla u|^{r(k-1)} |D^3u|^{r} \ dx\\
				&=: I_1 + I_2 . \\
			\end{split}
		\end{equation}
		
		We estimate the term $I_1$. By applying a standard Young inequality we get
		\begin{equation}
			\begin{split}
				I_1 &= C(k) \int_{B} |\nabla u|^{r(k-2)} |D^2u|^{2r} |D^2u| |D^2u|^{-1} |\nabla u|^{\frac{p-2-\beta}{2}} |\nabla u|^{\frac{2-p+\beta}{2}} \ dx\\
				&\le \frac{C(k)}{2} \int_{B} |\nabla u|^{p-2-\beta} |D^2u|^{2} \ dx  +  \frac{C(k)}{2} \int_{B} |\nabla u|^{2r(k-2)+2-p+\beta} |D^2u|^{4r-2} \ dx \\
			\end{split}
		\end{equation}
		
		Iterating this procedure $\tilde N$-times we obtain
		\begin{equation}
			\begin{split}
				I_1 & \le C(k,\tilde N) \int_{B} |\nabla u|^{p-2-\beta} |D^2u|^{2} \ dx \\
				&+  C(k,\tilde N) \int_{B} |\nabla u|^{2^{\tilde N}r(k-2)+\sum_{s=1}^{\tilde N}2^{s-1}(2-p+\beta)} |D^2u|^{\tilde q_{\tilde N}} \ dx, 
			\end{split}
		\end{equation}
		where $\tilde q_{\tilde N}$ is given by 
		
		\begin{equation} \label{definizionediq_n4}
			\begin{cases}
				\tilde q_{\tilde N} = 2(\tilde q_{\tilde N-1}-1) \\
				q_0 = 2r,\\
			\end{cases}
		\end{equation}
		and  $C(k,\tilde N)$ is a positive constant. 
		
		For $\beta \approx 1$, we get 
		\begin{equation}\label{mah}
			2^{\tilde N}r(k-2)+\sum_{s=1}^{\tilde N}2^{s-1}(2-p+\beta)\ge 0 \iff k>\frac{3-p}{2^{\tilde N}r}+\frac{p-3+2r}{r}.
		\end{equation}
		Now we fixed $k>(\alpha +1)/2$. We note that, under assumptions on $p$ (see \eqref{valorip}), we have $(\alpha+1)/2\ge p-1$.
		Recalling $r=1+\tilde \delta$, for any $\tilde N$ sufficiently large, there exists $\tilde \delta=\tilde \delta(\tilde N)>0$, sufficiently small such that $\tilde q_{\tilde N}\le 4.$ Moreover, by \eqref{mah}, for any $\tilde N$ sufficiently large, there exists $\tilde \varepsilon>0$ sufficiently small 
		\begin{equation}\label{mah2}
			k>p-1+\tilde\varepsilon.
		\end{equation}
		Using Theorem \ref{teosecond} and  Theorem \ref{teoremacalderon} we get
		\begin{equation}\label{eq:uno}
			I_1  \le C +  C \int_{B}  |D^2u|^{\tilde q_{\tilde N}} \,dx \le C,
		\end{equation}
		where $C$ is a positive constant not depending on $\epsilon.$ \newline
		We estimate the term $I_2.$ Using a Holder inequality with exponents $\left( \frac{2}{r}, \frac{2}{2-r} \right)$ we get
		\begin{equation}
			\begin{split}
				I_2 &= \int_{B} |\nabla u|^{r(k-\frac{1}{2}(p+\alpha))} |\nabla u|^{\frac{r}{2}(p-2+\alpha)} |D^3u|^{r} \ dx \\
				&\le \left( \int_{B} |\nabla u|^{\frac{2r}{2-r}(k-\frac{1}{2}(p+\alpha))} \ dx \right)^{\frac{2-r}{2}} \left( \int_{B} |\nabla u|^{p-2+\alpha} |D^3u|^{2} \ dx \right)^{\frac{r}{2}}. 
			\end{split}
		\end{equation}
		We note that for any $\tilde N$ sufficiently large there exist $\hat \varepsilon=\hat \varepsilon(\tilde N)$ sufficiently small, such that 
		\begin{equation}\label{mih1}
			\frac{2r}{2-r}(k-\frac{1}{2}(p+\alpha))>1-p \iff k>\frac{\alpha+1}{2}+\hat \varepsilon.
		\end{equation}
		
		By \eqref{mih1}, using Theorem \ref{inversodelpeso} and Theorem \ref{derivate terze}
		we get
		\begin{equation}\label{eq:due}
			I_2 \le C,
		\end{equation}
		where $C$ is a positive constant not depending on $\epsilon.$ 
		
		So by \eqref{eq:uno} and \eqref{eq:due}, for any $k > (\alpha+1)/2$, we have that 
		\begin{equation}
			\int_{B} \left| \frac{\partial V_{\epsilon,i}}{\partial x_j \partial x_l}  \right|^{r} \ dx \le C,
		\end{equation}
		where $C$ is a positive constant not depending on $\epsilon.$ \newline
		Since $W^{2,r}_{loc}(\Omega)$ is reflexive space, there exists 
		$\Tilde{V} \in W^{2,r}_{loc}(\Omega)$ such that $$V_{\epsilon,i} \rightharpoonup \Tilde{V} \quad \text{for  } \epsilon \rightarrow 0.$$
		By the compact embedding, $V_{\epsilon,i} \rightarrow \Tilde{V}$ in $L^{q}(\Omega)$ with $q<2^{*}$ and up to subsequence $V_{\epsilon,i} \rightarrow \Tilde{V}$ a.e. in $\Omega.$ 
		Since $V_{\epsilon,i} \rightarrow |\nabla u|^{k-1} u_i$ a.e. in $\Omega$, then $$\tilde V=|\nabla u|^{k-1} u_i  \in W^{2,r}_{loc}(\Omega)\subset  W^{2,1}_{loc}(\Omega).$$
		\newline
		In the case $f$ has no sign we cannot apply Theorem \ref{inversodelpeso}. So by similar computations we get
		\begin{equation}\label{eq:vvvv2}
			\begin{split}
				\int_{B} &\left| \frac{\partial V_{\epsilon,i}}{\partial x_j \partial x_l}  \right|^{r} \ dx \\
				&\le C(k) \int_{B} |\nabla u|^{r(k-2)} |D^2u|^{2r} \ dx 
				+ C(k) \int_{B} |\nabla u|^{r(k-1)} |D^3u|^{r} \ dx 
				=: I_1 + I_2 . \\
			\end{split}
		\end{equation}
		
		Now the term $I_1$ can be estimated as in the previous case. For the term $I_2$, from Theorem \ref{derivateterze2} and for $k\ge (p+\alpha)/2$, applying a Holder inequality with exponents $(\frac{2}{r}, \frac{2}{2-r})$ we get
		\begin{equation}
			\begin{split}
				I_2 &= \int_{B} |\nabla u|^{r(k-\frac{1}{2}(p+\alpha))} |\nabla u|^{\frac{r}{2}(p-2+\alpha)} |D^3u|^{r} \ dx \\
				&\le \left( \int_{B} |\nabla u|^{\frac{2r}{2-r}(k-\frac{1}{2}(p+\alpha))} \ dx \right)^{\frac{2-r}{2}} \left( \int_{B} |\nabla u|^{p-2+\alpha} |D^3u|^{2} \ dx \right)^{\frac{r}{2}}\le C, 
			\end{split}
		\end{equation}
		where $C$ is a positive constant not depending on $\varepsilon$. Proceeding as in the previous case we get the thesis.
	\end{proof}

	Using Theorem \ref{derivateterze2} we can prove

	\begin{proof}[Proof of Theorem \ref{stressfield}]
		Let $x_0 \in \Omega$ and $R>0$ such that $B := B_{2R}(x_0) \subset \subset \Omega.$ Let us consider $u_{\epsilon}\in u+W^{1,p}_0(\Omega)$ be a solution of 
		\begin{equation}\label{okk}
			\int_B(\epsilon +|\nabla u_{\epsilon}|^{2})^{\frac{p-2}{2}} \, (\nabla u_{\epsilon},\nabla \varphi) \,dx = \int_B f\varphi\,dx \quad \forall  \varphi\in C^{\infty}_c(\Omega).
		\end{equation} 
		We remark that if $f \in C^{1,\beta'}_{loc}(\Omega)$, by standard regularity result \cite{DiKaSc},\cite{GT} $u_{\epsilon} \in C^{3}(B).$ We fix $i=1,...,n$ and we use $\varphi_i \in C^{\infty}_{c}(\Omega)$ in \eqref{okk}. Integrating by parts we get
		
		\begin{equation}  \label{eq:primoordine}
			\begin{split}
				&\int_{B} (\epsilon + |\nabla u_{\epsilon}|^2)^{\frac{p-2}{2}} (\nabla u_{\epsilon,i}, \nabla \varphi)\,dx \\ &+(p-2) \int_B   (\epsilon + |\nabla u_{\epsilon}|^2)^{\frac{p-4}{2}} (\nabla u_{\epsilon}, \nabla u_{\epsilon,i}) (\nabla u_{\epsilon}, \nabla \varphi)\,dx = \int_{B} f_i \varphi\,dx.
			\end{split}
		\end{equation}
		
		We prove \eqref{eq:tesi} in $B_R(x_0)$. The result will follow by a covering argument.  
		
		For $i,k,l \in \{1,...,n\}$ let us define
		\begin{equation}
			\varphi := (\epsilon + |\nabla u_{\epsilon}|^2)^{\frac{(\Tilde{\alpha}-1)(p-2)}{2}} u_{\epsilon,i} |u_{\epsilon,kl}|^{\Tilde{\alpha}-2} \psi^2,
		\end{equation}
		where $\Tilde{\alpha} \ge 3$ and $\psi$ is the standard cut-off function defined in \eqref{eq:psi}.

		So
		\begin{equation} \label{eq:nabla}
			\begin{split}
				\nabla \varphi &= (\Tilde{\alpha} -1)(p-2) (\epsilon + |\nabla u_{\epsilon}|^2)^{\frac{(\Tilde{\alpha}-1)(p-2)-2}{2}} D^2u_{\epsilon} \nabla u_{\epsilon} u_{\epsilon,i} |u_{\epsilon,kl}|^{\Tilde{\alpha}-2} \psi^{2} \\
				&+ (\epsilon + |\nabla u_{\epsilon}|^2)^{\frac{(\Tilde{\alpha}-1)(p-2)}{2}} \nabla u_{\epsilon,i} |u_{\epsilon,kl}|^{\Tilde{\alpha}-2} \psi^{2}\\
				&+ (\Tilde{\alpha}-2) (\epsilon + |\nabla u_{\epsilon}|^2)^{\frac{(\Tilde{\alpha}-1)(p-2)}{2}} u_{\epsilon,i} |u_{\epsilon,kl}|^{\Tilde{\alpha}-3} \operatorname{sign}(u_{\epsilon,kl}) \nabla u_{\epsilon,kl} \psi^{2}\\
				&+ (\epsilon + |\nabla u_{\epsilon}|^2)^{\frac{(\Tilde{\alpha}-1)(p-2)}{2}} u_{\epsilon,i} |u_{\epsilon,kl}|^{\Tilde{\alpha}-2} 2 \psi \nabla \psi. \\
			\end{split}
		\end{equation}
		Substituting \eqref{eq:nabla} in \eqref{eq:primoordine} we get
		\begin{equation*}
			\begin{split}
				& (\Tilde{\alpha}-1)(p-2) \int_{B} (\epsilon + |\nabla u_{\epsilon}|^2)^{\frac{p-2}{2}} (\nabla u_{\epsilon,i}, D^2u_{\epsilon}\nabla u_{\epsilon}) \times \\
				& \qquad \qquad \qquad\qquad \qquad \qquad \qquad \times (\epsilon + |\nabla u_{\epsilon}|^2)^{\frac{(\Tilde{\alpha}-1)(p-2)-2}{2}} u_{\epsilon,i} |u_{\epsilon,kl}|^{\Tilde{\alpha}-2}\psi^2 \ dx \\
				&+ \int_{B} (\epsilon + |\nabla u_{\epsilon}|^2)^{\frac{p-2}{2}} |\nabla u_{\epsilon,i}|^{2} (\epsilon + |\nabla u_{\epsilon}|^2)^{\frac{(\Tilde{\alpha}-1)(p-2)}{2}} |u_{\epsilon,kl}|^{\tilde \alpha -2} \psi^2 \ dx\\
				&+(\Tilde{\alpha}-2) \int_{B} (\epsilon + |\nabla u_{\epsilon}|^2)^{\frac{p-2}{2}} (\nabla u_{\epsilon,i},\nabla u_{\epsilon,kl}) (\epsilon + |\nabla u_{\epsilon}|^2)^{\frac{(\Tilde{\alpha}-1)(p-2)}{2}} \times \\
				& \qquad \qquad \qquad \qquad \qquad\qquad \qquad \times u_{\epsilon,i} |u_{\epsilon,kl}|^{\Tilde{\alpha }-3} \operatorname{sign}(u_{\epsilon,kl}) \psi^2 \ dx  \\
                &+ 2\int_{B} (\epsilon + |\nabla u_{\epsilon}|^2)^{\frac{p-2}{2}} (\nabla u_{\epsilon,i}, \nabla \psi) (\epsilon + |\nabla u_{\epsilon}|^2)^{\frac{(\Tilde{\alpha}-1)(p-2)}{2}} u_{\epsilon,i} |u_{\epsilon,kl}|^{\tilde \alpha -2}  \psi \ dx \\
                &+ (\tilde \alpha-1)(p-2)^2 \int_{B} (\epsilon + |\nabla u_{\epsilon}|^2)^{\frac{p-4}{2}} (\nabla u_{\epsilon},\nabla u_{\epsilon,i}) (\nabla u_{\epsilon}, D^2u_{\epsilon}\nabla u_{\epsilon})\times \\ & \quad \quad \quad \quad \quad \quad \quad \quad \quad \quad \quad \quad \quad \quad \times (\epsilon + |\nabla u_{\epsilon}|^2)^{\frac{(\Tilde{\alpha}-1)(p-2)-2}{2}} u_{\epsilon,i} |u_{\epsilon,kl}|^{\tilde \alpha -2} \psi^2 \ dx\\
                \end{split}
                \end{equation*}
                \begin{equation*}
                \begin{split}
    &+ (p-2) \int_{B} (\epsilon + |\nabla u_{\epsilon}|^2)^{\frac{p-4}{2}} (\nabla u_{\epsilon},\nabla u_{\epsilon,i})^2 (\epsilon + |\nabla u_{\epsilon}|^2)^{\frac{(\tilde \alpha -1)(p-2)}{2}} |u_{\epsilon,kl}|^{\tilde \alpha -2}  \psi^2 \ dx \\
				&+ (\Tilde{\alpha}-2)(p-2) \int_{B} (\epsilon + |\nabla u_{\epsilon}|^2)^{\frac{p-4}{2}} (\nabla u_{\epsilon},\nabla u_{\epsilon,i}) (\nabla u_{\epsilon},\nabla u_{\epsilon,kl}) \times\\ & \quad \quad \quad \quad \quad \quad \quad \quad \quad \quad \quad \quad \times (\epsilon + |\nabla u_{\epsilon}|^2)^{\frac{(\Tilde{\alpha}-1)(p-2)}{2}} u_{\epsilon,i} |u_{\epsilon,kl}|^{\tilde \alpha-3} \operatorname{sign}(u_{\epsilon,kl}) \psi^2 \ dx \\
				&+ 2(p-2) \int_{B} (\epsilon + |\nabla u_{\epsilon}|^2)^{\frac{p-4}{2}} (\nabla u_{\epsilon},\nabla u_{\epsilon,i}) (\nabla u_{\epsilon},\nabla \psi) \times\\
				&\qquad \qquad \qquad \qquad\qquad \qquad \qquad \times (\epsilon + |\nabla u_{\epsilon}|^2)^{\frac{(\Tilde{\alpha}-1)(p-2)}{2}} u_{\epsilon,i} |u_{\epsilon,kl}|^{\tilde \alpha -2}\psi\,dx\\
				&= \int_{B} f_i (\epsilon + |\nabla u_{\epsilon}|^2)^{\frac{(\Tilde{\alpha}-1)(p-2)}{2}} u_{\epsilon,i} |u_{\epsilon,kl}|^{\tilde \alpha -2} \psi^2 \ dx.\\
			\end{split}
		\end{equation*}
		Summing over $i=1,...,n$, we get
		\begin{equation} \label{eq:sommai}
			\begin{split}
				& (\Tilde{\alpha}-1)(p-2) \int_{B} (\epsilon + |\nabla u_{\epsilon}|^2)^{\frac{\tilde \alpha (p-2)-2}{2}} |D^2u_{\epsilon}\nabla u_{\epsilon}|^2  |u_{\epsilon,kl}|^{\Tilde{\alpha}-2}\psi^2 \ dx \\
				&+ \int_{B} (\epsilon + |\nabla u_{\epsilon}|^2)^{\frac{\tilde \alpha (p-2)}{2}} |D^2u_{\epsilon}|^{2}  |u_{\epsilon,kl}|^{\tilde \alpha -2} \psi^2 \ dx\\
				&+(\Tilde{\alpha}-2) \int_{B} (\epsilon + |\nabla u_{\epsilon}|^2)^{\frac{\tilde \alpha (p-2)}{2}} (D^2u_{\epsilon} \nabla u_{\epsilon},\nabla u_{\epsilon,kl})  |u_{\epsilon,kl}|^{\Tilde{\alpha }-3} \operatorname{sign}(u_{\epsilon,kl}) \psi^2 \ dx  \\
				&+ 2\int_{B} (\epsilon + |\nabla u_{\epsilon}|^2)^{\frac{\tilde \alpha (p-2)}{2}} (D^2u_{\epsilon} \nabla u_{\epsilon}, \nabla \psi) |u_{\epsilon,kl}|^{\tilde \alpha -2}  \psi \ dx \\
            \end{split}
            \end{equation}
            \begin{equation*}
            \begin{split}
				&+ (\tilde \alpha-1)(p-2)^2 \int_{B} (\epsilon + |\nabla u_{\epsilon}|^2)^{\frac{\tilde \alpha (p-2)-4}{2}} (\nabla u_{\epsilon}, D^2u_{\epsilon}\nabla u_{\epsilon})^{2} |u_{\epsilon,kl}|^{\tilde \alpha -2} \psi^2 \ dx\\
				&+ (p-2) \int_{B} (\epsilon + |\nabla u_{\epsilon}|^2)^{\frac{\alpha (p-2)-2}{2}} \sum_{i=1}^{n} (\nabla u_{\epsilon},\nabla u_{\epsilon,i})^2 |u_{\epsilon,kl}|^{\tilde \alpha -2}  \psi^2 \ dx \\  
				&+ (\Tilde{\alpha}-2)(p-2) \int_{B} (\epsilon + |\nabla u_{\epsilon}|^2)^{\frac{\tilde \alpha(p-2)-2}{2}} (\nabla u_{\epsilon},D^2u_{\epsilon} \nabla u_{\epsilon}) (\nabla u_{\epsilon},\nabla u_{\epsilon,kl}) \times\\
				&\qquad \qquad \qquad \qquad \times |u_{\epsilon,kl}|^{\tilde \alpha-3} \operatorname{sign}(u_{\epsilon,kl}) \psi^2 \ dx \\
				&+ 2(p-2) \int_{B} (\epsilon + |\nabla u_{\epsilon}|^2)^{\frac{\tilde \alpha (p-2)-2}{2}} (\nabla u_{\epsilon},D^2u_{\epsilon}\nabla u_{\epsilon}) (\nabla u_{\epsilon},\nabla \psi)  |u_{\epsilon,kl}|^{\tilde \alpha -2}\psi\\
				&- \int_{B} (\nabla f,\nabla u_{\epsilon}) (\epsilon + |\nabla u_{\epsilon}|^2)^{\frac{(\Tilde{\alpha}-1)(p-2)}{2}}  |u_{\epsilon,kl}|^{\tilde \alpha -2} \psi^2 \ dx=:I_1+\cdot\cdot\cdot +I_9=0.\\
			\end{split}
		\end{equation*}
		
		Since $p<2$, using the Cauchy-Schwarz
		inequality we get
		\begin{equation} \label{eq:1}
			\begin{split}
				I_1 &=(\tilde \alpha -1)(p-2) \int_{B} (\epsilon + |\nabla u_{\epsilon}|^2)^{\frac{\tilde \alpha (p-2)-2}{2}} |D^2u_{\epsilon} \nabla u_{\epsilon}|^{2} |u_{\epsilon,kl}|^{\tilde \alpha-2} \psi^2 \ dx \\
				&\ge (\tilde \alpha -1)(p-2) \int_{B} (\epsilon + |\nabla u_{\epsilon}|^2)^{\frac{\tilde \alpha (p-2)}{2}} |D^2u_{\epsilon}|^{2} |u_{\epsilon,kl}|^{\tilde \alpha-2} \psi^2 \ dx \\
			\end{split}
		\end{equation}
		In the same way, we deduce
		\begin{equation} \label{eq:2}
			\begin{split}
				I_6 &=(p-2) \int_{B} (\epsilon + |\nabla u_{\epsilon}|^2)^{\frac{\tilde \alpha (p-2)-2}{2}} \sum_{i=1}^{n} (\nabla u_{\epsilon,i},\nabla u_{\epsilon})^2 |u_{\epsilon,kl}|^{\tilde \alpha -2} \psi^2 \ dx \\
				&\ge (p-2) \int_{B} (\epsilon + |\nabla u_{\epsilon}|^2)^{\frac{\tilde \alpha (p-2)}{2}} |D^2u_{\epsilon}|^2 |u_{\epsilon,kl}|^{\tilde \alpha -2} \psi^2 \ dx \\
			\end{split} 
		\end{equation}
		
		Using \eqref{eq:1} and \eqref{eq:2} in \eqref{eq:sommai} we get
		\begin{equation}
			\begin{split}
				&[(p-1)+(\tilde \alpha -1)(p-2)]\int_{B} (\epsilon + |\nabla u_{\epsilon}|^2)^{\frac{\tilde \alpha (p-2)}{2}} |D^2u_{\epsilon}|^2 |u_{\epsilon,kl}|^{\tilde \alpha -2} \psi^2 \ dx \\
				&\le I_1+I_2+I_5+I_6\\
				&\le (\tilde \alpha -2)(p-1)\int_{B} (\epsilon + |\nabla u_{\epsilon}|^2)^{\frac{\tilde \alpha (p-2)+1}{2}} |D^2u_{\epsilon}| |u_{\epsilon,kl}|^{\tilde \alpha -3} |\nabla u_{\epsilon,kl}| \psi^2 \ dx \\
				&+ (2+2|p-2|)\int_{B} (\epsilon + |\nabla u_{\epsilon}|^2)^{\frac{\tilde \alpha (p-2)+1}{2}} |D^2u_{\epsilon}| |u_{\epsilon,kl}|^{\tilde \alpha -2} |\nabla \psi| \psi \ dx \\ 
				&+ \int_{B} |\nabla f|  (\epsilon + |\nabla u_{\epsilon}|^2)^{\frac{(\Tilde{\alpha}-1)(p-2)+1}{2}}  |u_{\epsilon,kl}|^{\tilde \alpha -2} \psi^2 \ dx\\
			\end{split}
		\end{equation}
		
		Taking $p > 2 - \frac{1}{\tilde\alpha}$ we get
		\begin{equation} \label{eq:i1i2}
			\begin{split}
				&\int_{B} (\epsilon + |\nabla u_{\epsilon}|^2)^{\frac{\tilde \alpha (p-2)}{2}} |D^2u_{\epsilon}|^2 |u_{\epsilon,kl}|^{\tilde \alpha -2} \psi^2 \ dx \\
				&\le C(\tilde\alpha,p)\int_{B} (\epsilon + |\nabla u_{\epsilon}|^2)^{\frac{\tilde \alpha (p-2)+1}{2}} |D^2u_{\epsilon}| |u_{\epsilon,kl}|^{\tilde \alpha -3} |\nabla u_{\epsilon,kl}| \psi^2 \ dx \\
				&+ C(p)\int_{B} (\epsilon + |\nabla u_{\epsilon}|^2)^{\frac{\tilde \alpha (p-2)+1}{2}} |D^2u_{\epsilon}| |u_{\epsilon,kl}|^{\tilde \alpha -2} |\nabla \psi| \psi \ dx \\ 
				&+  \int_{B} |\nabla f|  (\epsilon + |\nabla u_{\epsilon}|^2)^{\frac{(\Tilde{\alpha}-1)(p-2)+1}{2}}  |u_{\epsilon,kl}|^{\tilde \alpha -2} \psi^2 \ dx\\
				&=: \Tilde{I}_1 + \Tilde{I}_2 + \Tilde{I}_3,\\
			\end{split}
		\end{equation}
		where $C(\tilde \alpha,p)$ and $C(p)$ are positive constants.
		
		Now we estimate the term $\Tilde{I}_1$ of \eqref{eq:i1i2}. From Proposition \ref{soluzioneregolarizzata} (see also Remark \ref{remarkino}), and using a standard Young inequality we get: 
		\begin{equation}
			\begin{split}
				\Tilde{I}_1 &\le \frac{C(\Tilde{\alpha}, p)}{2} \int_{B} (\epsilon + |\nabla u_{\epsilon}|^2)^{\frac{p-2+\alpha}{2}} |D^2u_{\varepsilon}|^{\gamma-1}|D^3u_\varepsilon|^2 \psi^2 \ dx \\
				&+ \frac{C(\Tilde{\alpha}, p)}{2} \int_{B} (\epsilon + |\nabla u_{\epsilon}|^2)^{\frac{2\tilde\alpha(p-2)+4-p-\alpha}{2}}   |D^{2}u_{\varepsilon}|^{2\tilde\alpha -3 -\gamma} \psi^2 \ dx \\ 
				&\le C(\gamma,\alpha,p,q,n,f,R ,\|\nabla u\|_{L^{\infty}_{loc}(\Omega)}, \Tilde{\alpha})\\
				&+ \frac{C(\Tilde{\alpha}, p)}{2} \int_{B} (\epsilon + |\nabla u_{\epsilon}|^2)^{\frac{2\tilde\alpha(p-2)+4-p-\alpha}{2}}   |D^{2}u_{\varepsilon}|^{2\tilde\alpha -3 -\gamma}  \ dx, \\  
			\end{split}
		\end{equation}
		where $\alpha, \gamma$ are given in Theorem \ref{derivateterze2}, and $C(\gamma,\alpha,p,q,n,f,R ,\|\nabla u\|_{L^{\infty}_{loc}(\Omega)}, \Tilde{\alpha})$ is a positive constant.

		From Theorem \ref{derivateterze2} (see also Remark \ref{remarkino}) we can choose $\alpha = \alpha (p,0)=4-p,$ and for $\beta \approx 1$, we obtain
		\begin{equation}
			2\tilde\alpha(p-2)+4-p-\alpha\ge p-2-\beta  \iff p>2-\frac{1}{2\tilde\alpha-1}.
		\end{equation}
		
		For $\gamma=2\tilde \alpha-5$, by Theorem \ref{teosecond}, we have that
		\begin{equation}\label{eq:I_1}
			\Tilde{I}_1  \le C(\gamma,\alpha, \beta, p,q,n,f,R,\|\nabla u\|_{L^{\infty}_{loc}(\Omega)}, \Tilde{\alpha})
		\end{equation}
		
		We note that, since $q\ge 3+\gamma=2(\tilde \alpha-1)\ge 4$ and  $p$ satisfies \eqref{assunzionisup}, we can apply Theorem \ref{derivateterze2}. Furthermore, from Theorem \ref{derivateterze2}, we note that $f\in W^{2,q/(q-\gamma)}_{loc}(\Omega)=W^{2,2(\tilde \alpha -1)/3}_{loc}(\Omega)$.

		We estimate the term $\Tilde{I}_2$. Since $p$ fulfils \eqref{assunzionisup}, by definition of $\psi$ and using Theorem \ref{teoremacalderon} we get
		\begin{equation} \label{eq:I_2}
			\begin{split}
				\Tilde{I}_2 &\le C(R,\tilde\alpha,p,\|\nabla u\|_{L^{\infty}_{loc}(\Omega)}) \int_{B} |D^2u_{\varepsilon}|^{\tilde\alpha-1} \ dx \\
				&\le C(p,q,n,f,R,\|\nabla u\|_{L^{\infty}_{loc}(\Omega)}, \Tilde{\alpha})
			\end{split}
		\end{equation}
		where $C(p,q,n,f,R,\|\nabla u\|_{L^{\infty}_{loc}(\Omega)}, \Tilde{\alpha})$ is a positive constant.
		
		For the term $\Tilde{I}_3$, by assumptions on $f$ and $p$, we deduce 
		\begin{equation} \label{eq:I_3}
			\begin{split}
				\Tilde{I}_3 &\le C(\tilde\alpha,p,\|\nabla u\|_{L^{\infty}_{loc}(\Omega)}) \int_{B} |\nabla f| |D^2u_{\epsilon}|^{\tilde\alpha -2} \ dx\\
				&\le C(\tilde\alpha,p,q,n,f,R,\|\nabla u\|_{L^{\infty}_{loc}(\Omega)}),
			\end{split}
		\end{equation}
		where $C(\tilde\alpha,p,q,n,f,R,\|\nabla u\|_{L^{\infty}_{loc}(\Omega)})$ is a positive constant.
		By \eqref{eq:I_1}, \eqref{eq:I_2} and \eqref{eq:I_3} we get
		
		\begin{equation}\label{stiimaa}
			\begin{split}
				\int_{B} (\epsilon + |\nabla u_{\epsilon}|^2)^{\frac{\tilde \alpha (p-2)}{2}} |D^2u_{\epsilon}|^2 |u_{\epsilon,kl}|^{\tilde \alpha -2} \psi^2 \ dx
				\le C(\gamma,\alpha,p,q,n,f,R,\|\nabla u\|_{L^{\infty}_{loc}(\Omega)} , \Tilde{\alpha}),\\
			\end{split}
		\end{equation}
		where $C(\gamma,\alpha,p,q,n,f,R,\|\nabla u\|_{L^{\infty}_{loc}(\Omega)} , \Tilde{\alpha})$ is a positive constant.
		
		Let us consider a compact set $K\subset\subset B\setminus Z_u$. By Schauder estimates we remark that
		$$\|u_\varepsilon\|_{C^{2,\tilde \beta}(K)}\le C,$$
		where $\tilde \beta \in (0,1)$ and $C$ is a positive constant not depending on $\varepsilon$.

		Setting now $$V_\varepsilon:=(\varepsilon+|\nabla u_\varepsilon|^2)^{\frac{p-2}{2}}\nabla u_\varepsilon,$$ by \eqref{stiimaa}, we deduce  $V_\varepsilon\in W_{loc}^{1,\tilde \alpha}(\Omega)$, and it is uniformly bounded in the space. There exists $\tilde V\in
		W_{loc}^{1,\tilde \alpha}(\Omega)$ such that
		$$V_\varepsilon\rightharpoonup \tilde V \quad \text{for }\varepsilon\rightarrow 0.$$
		By the compact embedding $V_\varepsilon \rightarrow \tilde V$ in $L^q_{loc}(\Omega)$ with $q < 2^*$ and up to subsequence
		$V_\varepsilon\rightarrow \tilde V$ a.e.. $V_\varepsilon\rightarrow |\nabla u|^{p-2}\nabla u$ a.e. therefore
		$$\tilde V=|\nabla u|^{p-2}\nabla u\in W_{loc}^{1,\tilde \alpha}(\Omega)$$
providing the desired regularity result. Let us also emphasize that, 
		by Fatou's Lemma, we also deduce that
		\begin{equation*}
			\int_{B_R(x_0)\setminus Z_u} |\nabla u|^{\Tilde{\alpha}(p-2)} |D^2u|^{\Tilde{\alpha}} \ dx \le C(\gamma,\alpha,p,q,n,f,R,\|\nabla u\|_{L^{\infty}_{loc}(\Omega)},  \Tilde{\alpha}),
		\end{equation*}
		where $C(\gamma,\alpha,p,q,n,f,R,\|\nabla u\|_{L^{\infty}_{loc}(\Omega)},  \Tilde{\alpha})$ is a positive constant. 
	\end{proof}

\end{document}